\documentclass[11pt,reqno,UTF8,twoside]{amsart}
\usepackage{mathrsfs}
\usepackage{amsfonts,amssymb,amsmath,amsthm}
\usepackage{cite}
\usepackage[colorlinks=true,citecolor=red]{hyperref}
\usepackage{titletoc}
\usepackage{geometry}
\usepackage{bm}
\usepackage{indentfirst}
\usepackage{graphicx}
\usepackage{float} 
\usepackage{booktabs}
\usepackage{longtable}
\usepackage{subfigure}
\usepackage{stmaryrd}
\usepackage{enumerate}
\usepackage{tikz}
\usepackage{ulem}
\usepackage{color,soul}
\usepackage{setspace}
\usepackage{stmaryrd}
\usepackage[pagewise]{lineno} 
\allowdisplaybreaks[2]
\usepackage{appendix}
\usepackage{cases}
\usepackage{todonotes}
\usepackage{enumitem}
\usepackage{comment}
\usepackage{graphicx}

\numberwithin{equation}{section}
\newtheorem{theorem}{Theorem}[section]
\newtheorem{lemma}[theorem]{Lemma}
\newtheorem{corollary}[theorem]{Corollary}
\newtheorem{proposition}[theorem]{Proposition}

\newtheorem{claim}[theorem]{Claim}
\newtheorem{definition}[theorem]{Definition}
\theoremstyle{remark}

\newtheorem{rmk}[theorem]{Remark}

\newcommand{\R}{\mathbb{R}}

\hypersetup{linkcolor=blue}

\makeatletter
\@namedef{subjclassname@2020}{\textup{2020} Mathematics Subject
Classification}
\makeatother

\begin{document}
    
\title[Mixing for the Chirikov map]{Quantitative Exponential Mixing for the Randomized Chirikov Standard Map}

\author[Z. Liu]{Ziyu Liu}
\address[Ziyu Liu]{School of Mathematics and Physics, University of Science and Technology Beijing, 100083, Beijing, China.}
\email{ziyu@ustb.edu.cn}

\author[Y. Shi]{Yankai Shi}
\address[Yankai Shi]{School of Mathematical Sciences, Peking University, 100871, Beijing, China.}
\email{shiyankai@stu.pku.edu.cn}

	

\begin{abstract}
We investigate the mixing properties of a randomized  Chirikov standard map on $\mathbb{T}^2$. While the deterministic dynamics exhibit obstructions to global ergodicity, we establish  explicit  almost-sure quantitative exponential mixing when  kicking strengths are sufficiently large. To achieve this, we formulate a criterion for incompressible random dynamical systems, reducing quantitative exponential mixing to serval verifiable conditions. Additionally, we provide a milder parameter condition to derive qualitative exponential mixing and enhanced dissipation.
\end{abstract}

\maketitle
\setcounter{tocdepth}{1}

\tableofcontents

\section{Introduction}\label{sec: introduction}
Mixing by incompressible flows is a fundamental mechanism of major importance in natural sciences and engineering. Physically, it transfers scalar fluctuations from macroscopic spatial scales to microscopic ones. In purely advective regimes, this transfer is conservative and reversible over finite times; however, long-time evolution with diffusion causes irreversible information loss and promotes equilibration of temperature or chemical concentration. Thus, understanding mixing is crucial for studying passive scalar dynamics and turbulence.

\vspace{0.5em}

From a mathematical perspective, mixing is typically described from two viewpoints. In the Eulerian framework, the relevant  model is the advection-diffusion equation
\begin{equation}
    \label{eq:advection-diffusion}
    \begin{cases}
        \partial_t \rho + u \cdot \nabla \rho = \nu \Delta \rho, \\
        \rho \big|_{t=0} = \rho_0,
    \end{cases}
\end{equation}posed on a periodic domain $\mathbb{T}^2=\mathbb{R}^2/(2\pi\mathbb{Z})^2$. Here, $u\colon[0,+\infty)\times\mathbb{T}^2\rightarrow\mathbb{R}^2$ is a given divergence-free velocity field, $\nu\geq 0$ is the diffusivity constant. The solution $\rho$ to \eqref{eq:advection-diffusion} represents the concentration of a passive scalar and can be assumed to be mean-zero:
\begin{equation*}
    \int_{\mathbb{T}^2}\rho_t(x)\mathrm{d}x=0,\qquad \forall\,t\geq0.
\end{equation*}
Following \cite{MMP-05,Thiffeault-11}, mixing can usually be quantified by the negative Sobolev norms of $\rho_t$, i.e., 
\begin{equation}\label{negative sobolev}
    \|\rho_t\|_{\dot{H}^{-s}}^2:=\sum_{k\in\mathbb{Z}^2\setminus\{(0,0)\}}\left|k\right|^{-2s}\left|\hat{\rho}_k(t)\right|^2, \qquad s>0,
\end{equation}whose decay implies that the $L^2$-mass of the scalar is being pushed toward high frequency modes.

Meanwhile, from the Lagrangian viewpoint, the flow map $\phi_t$ is generated by the velocity field $u$, which is given by
\begin{equation}
    \label{eq:Lagrangian flow}
    \begin{cases}
       \mathrm{d}{\phi}_t(x) = u(t,\phi_t(x))\mathrm{d}t+\sqrt{2\nu}\mathrm{d}W_t,\\
        \phi_0(x)=x.
    \end{cases}
\end{equation}Here $\phi_t(x)\in\mathbb{T}^2$ denotes the position of a passive particle at time $t$. Indeed, by Feynman--Kac formula, the solution $\rho$ of \eqref{eq:advection-diffusion} can be represented via the Lagrangian flow $\phi_t$ by
\begin{equation*}
    \rho_t(x)=\mathbb{E}_W\!\left[\rho_0\!\circ\phi_t^{-1}(x)\right],
\end{equation*}where $\mathbb{E}_W$ denotes the expectation with respect to the Brownian motion $W_t$.

\vspace{0.5em}

Under this formulation, the study of mixing properties of \eqref{eq:advection-diffusion} can be reduced to the analysis of the flow map \eqref{eq:Lagrangian flow}. In the present paper, we use this Lagrangian framework to investigate the quantitative mixing properties of a randomized Chirikov standard map, as specified below.

\subsection{Model and main results}
A natural model for studying mixing on $\mathbb{T}^2$ is the Chirikov standard map, introduced in \cite{Chirikov-71,Chirikov-79}. It is a fundamental model in physics and dynamical systems, describing the transition from integrable motion to chaos in Hamiltonian systems. Originating from the periodically kicked rotor, it is defined on $\mathbb{T}^2$ by
\begin{equation*}
    \begin{cases}
        x_{n+1,1}=x_{n,1}+K\sin(x_{n,2}),\\
        x_{n+1,2}=x_{n,2}+x_{n+1,1},
    \end{cases}
\end{equation*}
where $x_{n,1}$ and $x_{n,2}$ denote the momentum and angle, respectively, both taken modulo $2\pi$, and $K>0$ is the kicking strength.

In the deterministic setting, global mixing is obstructed by persistent regular structures. For small $K$, invariant Kolmogorov--Arnold--Moser~(KAM) curves confine the momentum $x_{n,1}$ and preclude global ergodicity. As $K$ increases, Greene~\cite{Greene-79} identified $K_g\approx0.9716$ as the threshold for the destruction of the last invariant curve, though its relation to global momentum transport remains open; see \cite{Greene-79,Mackay-83,Mackay-84}. Even for large $K$, the deterministic dynamics still exhibits a complex mixture of behaviors: Duarte~\cite{Duarte-94} proved the existence of dense microscopic elliptic islands for a residual set of parameters, while Gorodetski~\cite{Gorodetski-12} constructed a transitive stochastic sea of full Hausdorff dimension for another residual set. The persistence of elliptic islands shows that the unperturbed standard map cannot be expected to achieve the global mixing.

\vspace{0.5em}

Motivated by the deterministic case, we study whether randomness can induce mixing in this system. In view of the complexity of the dynamics, we are particularly interested in obtaining quantitative estimates for the resulting mixing rate. We now introduce the randomized Chirikov standard map considered in this paper.

Let $\{\omega_i=(\omega_i^1,\omega_i^2)\}_{i\in\mathbb{N}}$ be a sequence of i.i.d random variables uniformly distributed on $[0,2\pi)^2$. For $(t,x)\in[0,+\infty)\times \mathbb{T}^2$, the dynamics is generated by a time-dependent, divergence free velocity field $u$ with parameter $K$, defined piecewise as 
\begin{align}
    u(t,x)&=\begin{pmatrix}
    K\sin(x_2-\omega_n^1)\label{velocity 1}\\
    0
\end{pmatrix}, \qquad t\in [2n-2,2n-1),\\
u(t,x)&=\begin{pmatrix}
    0\\
    x_1-\omega_n^2
\end{pmatrix}, \qquad\qquad\quad t\in [2n-1,2n),\label{velocity 2}
\end{align}
where all additions are understood modulo $2\pi$.

Given $\omega=\left(\omega^1,\omega^2\right)\in [0,2\pi)^2$ and an initial data $x=(x_1,x_2)\in \mathbb{T}^2$, the horizontal and vertical flow maps are defined respectively by
\begin{equation*}
    f_{\omega}^{H}(x):=\begin{pmatrix}
        x_1 + K\sin(x_2 - \omega^1) \\
        x_2
    \end{pmatrix},\qquad 
    f_{\omega}^{V}(x):=\begin{pmatrix}
        x_1 \\
        x_2 + x_1-\omega^2
    \end{pmatrix}.
\end{equation*}
Their composition yields the discrete-time random dynamical system
\begin{equation}
    \label{eq:discrete-map}
    f_{\omega}(x):=f_{\omega}^{V} \circ f_{\omega}^{H}(x).
\end{equation}This can be interpreted as a randomized Chirikov map with independent random phases $\omega_i^1$ and $\omega_i^2$ at each iteration. Let $\phi_t$ be the Lagrangian flow generated by $u$. Note that
\begin{equation*}
    \phi_2(x)=f_{\omega_1}(x),
\end{equation*}the particle trajectories at integer times $t=2n$ are recovered by 
\begin{equation*}
    \phi_{2n}(x)=f_{\underline{\omega}}^n(x):=f_{\omega_n}\circ\cdots\circ f_{\omega_1}(x),
\end{equation*}where $\underline{\omega}=(\omega_1,\omega_2,\ldots)\in\left([0,2\pi)^2\right)^{\mathbb{N}}=:\Omega$ denotes a realization of the random sequence. 

\vspace{0.5em}
  Our main result is  the following.
\begin{theorem}\label{thm:quantitative}
    Let $\phi_t$ denote the Lagrangian flow of equation \eqref{eq:Lagrangian flow} with $\nu=0$, velocity field $u$ given in \eqref{velocity 1},\eqref{velocity 2}, and let $p\in\!\left(0,\frac{1}{2}\right)$ be a fixed constant. Then, there exist constants $K_0,C_1,C_2>0$ such that for any $q>0$ and $K\geq K_0$, one can find a random variable $\hat D\colon\Omega\to[0,\infty)$ satisfying, for all mean-zero $\varphi,\psi\in H^{1}(\mathbb{T}^2)$,
    \begin{equation*}
        \left|\int_{\mathbb{T}^2}\varphi(x)\psi\circ f_{\underline{\omega}}^n(x)\mathrm{d}x\right|\leq \hat{D}(\underline{\omega})\exp{\!\left(-\frac{C_1}{1+q}e^{-K^{265}}n\right)}\|\varphi\|_{\dot{H}^1}\|\psi\|_{\dot{H}^1}.
    \end{equation*}Moreover, $\hat D$ satisfies the moment bound
    \begin{equation*}
        \mathbb{E}[\hat{D}^q]\leq C_2K^p.
    \end{equation*}
\end{theorem}

Theorem \ref{thm:quantitative} gives a quantitative mixing estimate, uniformly for all $K\geq K_0$. In contrast with the deterministic setting, it provides an explicit exponential rate together with a moment bound on the random prefactor. This indicates that the added randomness can overcome the stable structures present in the unperturbed dynamics.

Our proof relies on a general criterion for quantitative exponential mixing of Lagrangian flows, stated as Theorem \ref{thm: general framework for quantitative exponential mixing}. This criterion is not specific to the randomized Chirikov standard map: it can also be applied to the Pierrehumbert model considered in \cite{BCZG-23,CIS-25,son-25}, with details given in Appendix \ref{sec: Application to the Pierrehumbert model}. We also expect that the arguments can be used for the ABC flow studied in \cite{coti-26}.

\vspace{0.5em}
The deterministic theory suggests that the dependence of the dynamics on $K$ is quite delicate, due to the possible coexistence of regular and chaotic structures. The large-parameter condition $K\geq K_0$ in Theorem \ref{thm:quantitative} is used to obtain a detailed quantitative control of this regime. Even without tracking the explicit parameter dependence of the rate, we also derive the following global mixing result: almost-sure exponential mixing already holds for $K\geq 4\pi$.

\begin{theorem}\label{thm: exponential mixing}
    Assume $K\geq 4\pi$. Let $\phi_t^\nu$ denote the Lagrangian flow of equation \eqref{eq:Lagrangian flow} with $\nu\geq 0$, velocity field $u$ given in \eqref{velocity 1},\eqref{velocity 2}. Then, for any $q,s>0$ and all sufficiently small $\nu$, there exist a random variable $D_\nu\colon \Omega\rightarrow [0,+\infty)$ and a $\nu$-independent constant $\lambda_s>0$ such that for all mean-zero $g,h\in H^s(\mathbb{T}^2)$, the following holds
    \begin{equation*}
        \left|\int_{\mathbb{T}^2}g(x)h\circ \phi_t^\nu(x)\mathrm{d}x\right|\leq D_\nu(\underline{\omega})e^{-\lambda_s t}\|g\|_{\dot{H}^s}\|h\|_{\dot{H}^s},
    \end{equation*}almost surely for all $t>0$. Furthermore, there exists a $\nu$-independent constant $D_q>0$ such that 
    \begin{equation*}
        \mathbb{E}\left[D_\nu^q\right]\leq D_q.        
    \end{equation*}
\end{theorem}
\begin{rmk}
    By following the proof of Theorem \ref{thm: exponential mixing}, the discrete-time estimate in Theorem~\ref{thm:quantitative} can also be extended to the continuous-time setting $t>0$. We also expect the conclusion of Theorem~\ref{thm:quantitative} to extend to the viscous case $\nu>0$ without essential difficulty; for simplicity, this paper is restricted to the inviscid case $\nu=0$.
\end{rmk}  

The following result for passive scalar $\rho_t$ in \eqref{eq:advection-diffusion} is an immediate corollary of Theorem \ref{thm: exponential mixing}.
\begin{corollary}
    In the setting of {\rm Theorem \ref{thm: exponential mixing}}, for any mean-zero initial $\rho_0\in H^s(\mathbb{T}^2)$, then the solution $\rho_t$ to the advection-diffusion equation \eqref{eq:advection-diffusion} satisfies that
    \begin{itemize}
        \item[\tiny$\bullet$] Exponential mixing {\rm(}$\nu\geq 0${\rm):}
        \begin{equation*}
            \|\rho_t\|_{\dot{H}^{-s}}\leq D_\nu(\underline{\omega})\|\rho_0\|_{\dot{H}^{-s}}e^{-\lambda_s t},\quad \forall\,t\geq 0.
        \end{equation*}
        \item[\tiny$\bullet$] Enhanced dissipation {\rm(}$\nu>0${\rm):}
        \begin{equation*}
            \quad\; \|\rho_t\|_{L^2}\leq \nu^{-1-s}D_\nu(\underline{\omega})\|\rho_0\|_{L^2}e^{-\lambda_s t},\quad \forall\,t\geq 0.
        \end{equation*}
    \end{itemize}
\end{corollary}

\subsection{Strategy of the proof} We now outline the ingredients in the proof of Theorem~\ref{thm:quantitative}. Utilizing the duality representation and the push-forward property of Lagrangian flow $\rho_t=\rho_0\circ\phi_t^{-1}$, the decay of $\rho_t$ in negative Sobolev norms is reduced to quenched correlation estimates:
\begin{equation*}
    \|\rho_{2n}\|_{H^{-s}}=\sup_{\|\varphi\|_{\dot{H}^s}=1} 
    \int_{\mathbb{T}^2} \varphi(x)\rho_{2n}(x)\mathrm{d}x 
    = \sup_{\|\varphi\|_{\dot{H}^s}=1} \int_{\mathbb{T}^2} 
    \rho_0(x)\varphi\circ f_{\underline{\omega}}^n(x)\mathrm{d}x.
\end{equation*}
Consequently, the task is to establish quantitative correlation bounds for the randomized Chirikov map. Our strategy proceeds in two steps. We first formulate a quantitative criterion for incompressible random dynamical systems, and then apply it to the Chirikov standard map.

\subsubsection{Probabilistic criterion} Based on results from the ergodic theory (see e.g., \cite{HM-08,HM-11}), our mixing criterion reduces to two conditions: 
\begin{itemize}
    \item [(i)] uniform contraction estimates \eqref{eq: general uniform contraction};
    \item [(ii)] the existence of a small set.
\end{itemize}
For precise formulations, see in Theorem \ref{thm: general framework for quantitative exponential mixing}. Conceptually, the uniform contraction estimate serves as a probabilistic analogue to the existence of a spectral gap. Provided suitable derivative bounds, this condition enables us to construct the explicit Lyapunov function
\begin{equation*}
    V(x,y):=\operatorname{dist}(x,y)^{-p},\qquad x\neq y.
\end{equation*}Combining this with the small set property, the quantitative Harris theorem guarantees the ergodicity of the associated two-point process, from which the exponential mixing follows via standard arguments.
\begin{rmk}
    In the approaches of \cite{BBPS-22A,BCZG-23,coti-26}, the strict positivity of Lyapunov exponents plays a central role: it is used to obtain chaotic behavior of the Lagrangian dynamics, from which exponential mixing is then derived. The criterion used in this paper takes a different route. It does not require the positivity of Lyapunov exponents, nor does it use chaos as an intermediate step. Instead, it relies directly on probabilistic uniform estimates for the iterates $f_{\underline{\omega}}^n$, which are sufficient for proving quantitative mixing in the present setting.
\end{rmk}

\subsubsection{Verification for the randomized Chirikov map}\mbox{}\\
\noindent{\bf Uniform contraction.} Due to the degenerate action of the random phases $\omega$ in the randomized Chirikov map, the single iteration 
\begin{equation*}
    D_xf_{\omega_1}=\begin{pmatrix}
        1 & K\cos\!\left(x_2-\omega_1^1\right)\\
        1 & 1+K\cos\!\left(x_2-\omega_1^1\right)
    \end{pmatrix}
\end{equation*}
fails to strictly contract all tangent directions. To address this issue, we analyze the two-step Jacobian $D_xf_{\underline{\omega}^2}$. The compounded randomness from $\omega_1^1$ and $\omega_2^1$ effectively mixes the coordinates, which in turn yields the uniform contraction condition.

\vspace{0.3em}
\noindent{\bf Small set condition.} Verification for the small set condition constitutes the main difficulty in our setting, owing to the degenerate noise and the global dynamics of the system. We resolve this problem in two main steps. The first step establishes a local minorization near a specific reference point $z_*$. Despite the noise degeneracy, the four-step Jacobian evaluated at $z_*$ achieves full rank. This allows us to invoke the quantitative inverse and implicit function theorems, thereby yielding a local small set condition within a small neighborhood of $z_*$:
\begin{equation*}
    \inf_{z\in B(z_*,r(K))}P^{(2),n}(z,\cdot)\geq c(K)^n\mu(\cdot),\qquad c(K)\in (0,1).
\end{equation*}

To globalize this local minorization, we proceed to prove the topological irreducibility of the two-point process, a property that plays an important role in ergodic theory; see \cite{DPZ-96,HM-11EJP}. Irreducibility guarantees that, with strictly positive probability, the system can reach the small neighborhood of $z_*$ from any initial state. We accomplish this by translating the probabilistic problem into a global approximate controllability problem, see, e.g. \cite{Jurdjevic-97}. Specifically, the essential mechanism is to design an explicit sequence of phases $\{\omega_i\}$ that effectively steers the two-point trajectory into the local small set centered at $z_*$
\begin{equation*}
    \operatorname{dist}\!\left((f_{\underline{\omega}^{N_1}}(x),f_{\underline{\omega}^{N_1}}(y)),z_*\right)<r(K),\qquad N_1\leq C(K).
\end{equation*} The details are in Section \ref{subsubsec: Topological irreducibility of the two-point process}. Consequently, by coupling this global controllability with the local minorization, we deduce the desired global small set condition.

\subsection{Review of the literature}
A central objective in the study of mixing is to quantify the decay rate of a passive scalar under incompressible transport. For sufficiently regular velocity fields, such as $u\in L_t^\infty W_x^{1,p}$ with $p>1$ or uniformly Lipschitz fields, the mixing rate is at most exponential. Quantitative lower bounds in this regime were studied in \cite{CDL-08,IKX-14,LTD-11,Seis-13}, while the borderline case $u\in L_t^\infty W_x^{1,1}$ remains open and is closely related to Bressan's rearrangement cost conjecture \cite{Bressan-03}. This naturally leads to the question of which velocity fields can realize the optimal exponential rate. For prescribed initial data $\rho_0$, exponential mixing can be achieved by suitably designed incompressible flows \cite{ACM-19A,YZ-17}, but such constructions are highly datum-dependent. This issue was addressed in \cite{EZ-19}, where a single universal mixer was constructed for $\rho_0\in L^\infty\cap H^{\sigma}$.

Randomness provides another robust mechanism for exponential mixing. Early works \cite{Furstenberg-63,Ledrappier-86} established positivity of Lyapunov exponents for stochastic flows of diffeomorphisms, which later led to almost-sure exponential mixing for incompressible SDEs \cite{BS-88,DKK-04}. More recently, Bedrossian--Blumenthal--Punshon-Smith \cite{BBPS-22A} proved exponential mixing for scalars advected by stochastic Navier--Stokes equations, with further extensions in \cite{CR-26}. Their approach combines positivity of Lyapunov exponents \cite{BBPS-22B} with infinite-dimensional ergodic theory for SPDEs \cite{HM-06,HM-11EJP}. A dynamics-based framework was also introduced in \cite{BCZG-23}, and applied to the Pierrehumbert model \cite{BCZG-23,CIS-25,son-25} and randomized ABC flows \cite{coti-26} to obtain quenched exponential mixing.

Passive scalar mixing is closely related to enhanced dissipation. For $\nu>0$ in \eqref{eq:advection-diffusion}, the standard energy estimate gives
\begin{equation*}
    \|\rho_t\|_{L^2}\lesssim e^{-\nu t}\|\rho_0\|_{L^2}.
\end{equation*}
It is therefore natural to ask whether mixing-induced small scales can accelerate this decay. Enhanced dissipation was first established for time-independent flows in \cite{CKRZ-08}, with refined results for time-periodic fields in \cite{KSZ-08}. Since then, this phenomenon has been studied for shear flows, cellular flows, vortices and related structures; see, for example, \cite{CZEW-20,IZ-23,WZZ-20,Gallay-18,Wei-21}. In random settings, almost-sure enhanced dissipation for stochastic fluid models was proved in \cite{BBPS-21ED} via the ergodicity of the $\nu$-dependent two-point process $P^{(2),\nu}$. More recently, \cite{CIS-25} showed that, under suitable Harris conditions, enhanced dissipation can be derived from the inviscid two-point process, with applications to the Pierrehumbert model. In parallel, random cellular flows were treated in \cite{NFS-26} through an Eulerian approach based on deterministic hypoelliptic operators.

\vspace{0.5em}
\subsection*{Organization of the paper} The  paper is organized as follows. Section \ref{sec:preliminary} reviews some backgrounds on Markov processes and random dynamical systems. In Section \ref{sec: Quantitative exponential mixing for incompressible RDS}, we formulate a general criterion that deduces quantitative exponential mixing from a set of verifiable conditions. Section \ref{sec: Quantitative Mixing Rate: Proof of Theorem 2} is then dedicated to applying this criterion to the randomized Chirikov standard map.  Section \ref{sec: Mixing and enhanced dissipation: proof of Theorem 1.6} concludes the paper by establishing qualitative exponential mixing and enhanced dissipation under a relaxed parameter threshold.  Finally, Appendix collects some technical proofs and materials.

\vspace{0.5em}
For notation, given any prescribed sequence $\{\omega_i\}_{i\ge 1} \subset \Omega_0$, denote
\begin{equation*}
    \underline{\omega}^n=(\omega_1,\ldots,\omega_n)\in\Omega_0^n,\qquad \underline{\omega}=(\omega_1,\ldots,\omega_n,\ldots)\in\Omega_0^{\mathbb{N}}.
\end{equation*}Accordingly, the $n$-step composition of $f_\omega$ is defined as 
\begin{equation*}
    f_{\underline{\omega}}^n(x)=f_{\underline{\omega}^n}(x):=f_{\omega_n}\circ\cdots\circ f_{\omega_1}(x).
\end{equation*}

Let $\mathbb{S}^{d-1}$ denote the unit sphere in $\mathbb{R}^d$. Given $x\in\mathbb{T}^2$, $v\in\mathbb{S}^1$, a phase realization $\underline{\omega}^n\in\Omega_0^n$, and $f_\omega$ as in \eqref{eq:discrete-map}, the trajectories are written as
\begin{equation*}
    x_n:=f_{\underline{\omega}^n}(x),\quad w_n:=D_xf_{\underline{\omega}^n}v,\quad v_n:=\frac{w_n}{\left|w_n\right|}.
\end{equation*}For $i=1,2$, we let $x_{n,i}$ (respectively $w_{n,i}$ and $v_{n,i}$) denote the $i$-th coordinate of $x_n$ (respectively $w_n$ and $v_n$). The notation also extends 
to the case $x\in\mathbb{T}^d,\ v\in\mathbb{S}^{d-1}$ and $f_\omega\colon\mathbb{T}^d\rightarrow \mathbb{T}^d$.

\section{Preliminaries}\label{sec:preliminary}
In this section, we recall several fundamental concepts from the theory of Markov processes and random dynamical systems that are essential for our subsequent analysis. Letting $f_\omega$ denote the random dynamical system generated by \eqref{eq:discrete-map}, we begin by introducing three distinct Markov transition kernels induced by $f_\omega$.

\medskip
\noindent\textbf{One-point process.} This chain describes the position of a particle started at $x\in\mathbb{T}^2$ after one full iteration of the random flow map. Specifically, the Markov transition kernel $P$ is defined by 
\begin{equation}\label{eq:one-point}
    P(x,A)=\mathbb{P}(f_\omega(x)\in A),
\end{equation}where $A\in\mathcal{B}\left(\mathbb{T}^2\right)$ is a Borel set.

\medskip
\noindent\textbf{Projective process.}
For a particle initialized at $x\in\mathbb{T}^2$ with direction $v\in\mathbb{S}^1$, this process describes the joint evolution of the particle's position and direction. For this consider any $\hat{A}\in\mathcal{B}\left( \mathbb{T}^2\times\mathbb{S}^1\right)$, the Markov transition kernel $\hat{P}$ is defined by
\begin{equation*}
    \hat{P}\!\left((x,v),\hat{A}\right)=\mathbb{P}\!\left(\left(f_\omega(x),\frac{D_xf_\omega v}{\left|D_xf_\omega v\right|}\right)\in\hat{A}\right).
\end{equation*}

\medskip
\noindent\textbf{Two-point process.}
This process is defined on $\left(\mathbb{T}^2\times\mathbb{T}^2\right)\setminus\Delta$ where $\Delta:=\{(x,x):x\in\mathbb{T}^2\}$ denotes the diagonal in $\mathbb{T}^2\times\mathbb{T}^2$. The two-point chain describes the joint evolution of two particles starting from distinct initial positions $x,y \in \mathbb{T}^2$. Given any Borel set $A\in\mathcal{B}\left(\left(\mathbb{T}^2\times\mathbb{T}^2\right)\setminus\Delta\right)$, the corresponding transition kernel $P^{(2)}$ is defined by
\begin{equation}\label{eq: def of two-point process}
    P^{(2)}\!\left((x,y),A\right)=\mathbb{P}\left(\left(f_\omega(x),f_\omega(y)\right)\in A\right).
\end{equation}
The following section is devoted to the ergodic theory of Markov transition kernel $P$, which will be applied to all three different Markov chains introduced above. To present a unified framework, we denote the state space by $X$, assumed to be a complete metric space.

\subsection{Markov chain preliminaries}

\vspace{0.3em}
Let $\mathcal{B}(X)$ denote the Borel $\sigma-$algebra and $\mathcal{P}(X)$ the space of probability measures on $X$. For all $n\in\mathbb{N}$ and $A\in\mathcal{B}(X)$ , iterates of the Markov transition kernel $P$ are defined inductively by the Chapman–Kolmogorov equations
\begin{equation*}
    P_{n+1}(x,A)=\int_X P(x,\mathrm{d}y)P_n(y,A).
\end{equation*}
We use standard notation for the associated Markov semigroup $\{P_n\}_{n\in\mathbb{N}}$ acting on bounded measurable functions $P_n\colon B_b(X)\rightarrow B_b(X)$ and for its dual $P_n^*\colon \mathcal{P}(X)\rightarrow \mathcal{P}(X)$, defined by
\begin{equation*}
    P_nf(x)=\int_Xf(y)P_n(x,\mathrm{d}y), \quad P_n^*\mu(A)=\int_X P_n(x,A)\mu(\mathrm{d}x),
\end{equation*}
where $f\in B_b(X),\ \mu\in\mathcal{P}(X)$ and $A\in\mathcal{B}(X)$. Moreover, the Markov kernel $P$ is said to satisfy the \textit{Feller property} if $P f(x) \in C_b(X)$ for every $f\in C_b(X)$. The following notions are commonly used in the study of ergodicity.
\begin{definition}
    A probability measure $\pi\in\mathcal{P}(X)$ is said to be \textit{stationary} for Markov transition kernel $P$ if 
    \begin{equation*}
        \pi(A)=\int_X P(x,A)\pi(\mathrm{d}x),
    \end{equation*}for any $A\in\mathcal{B}(X)$.
\end{definition}
\begin{definition}
    Let $\pi\in\mathcal{P}(X)$. $A\in\mathcal{B}(X)$ is $(P,\pi)$-invariant if $P\mathbf{1}_A=\mathbf{1}_A$ holds $\pi-$a.e.. We say that $\pi$ is an ergodic stationary measure if all $(P,\pi)$-invariant sets $A$ satisfy $\pi(A)=0$ or $\pi(A)=1$.
\end{definition}

Next, we introduce a quantitative version of Harris theorem. For a detailed proof of this result, we refer to \cite{HM-11}. To this end, we begin by defining a family of weighted norms. Given a measurable function $V\colon X\rightarrow [0,+\infty)$ and a parameter $\beta>0$, let $L_\beta^{\infty}(X)$ be the space of all measurable functions $\phi$ satisfying
\begin{equation}\label{eq: beta-norm}
    \left\|\phi\right\|_{\beta}:=\sup_{x\in X}\frac{\left|\phi(x)\right|}{1+\beta V(x)}<+\infty.
\end{equation}

\begin{theorem}[Quantitative Harris theorem]\label{thm: quantitative harris}
    Let $P$ be a Feller transition kernel and assume the following:
    \begin{itemize}
        \item [(1)] (Lyapunov drift condition) There exists a function $V\colon X\rightarrow [0,+\infty)$ and constants $C\geq 0$ and $\gamma\in (0,1)$ such that
        \begin{equation*}
            PV(x)\leq \gamma V(x)+C,
        \end{equation*}for all $x\in X$.
        \item [(2)] (Quantitative small set condition) Let $V$ be the function in (1). There exists a constant $\alpha \in(0,1)$ and a probability measure $\nu\in\mathcal{P}(X)$ such that
        \begin{equation*}
            \inf_{x\in E}P(x,\cdot)\geq \alpha\nu(\cdot),
        \end{equation*}where $E=\{x\in X:V(x)\leq R\}$ for some $R>\frac{2C}{1-\gamma}$. Here $C$ and $\gamma$ are the constants from Assumption (1).
     \end{itemize}
     Then the transition kernel $P$ admits a unique stationary measure $\pi$, and there exist constants $\bar{\alpha}\in (0,1)$ and $\beta>0$ such that 
     \begin{equation*}
        \left\|P^n\phi-\int_X\phi(x)\pi(\mathrm{d}x)\right\|_\beta\leq\bar{\alpha}^n\left\|\phi-\int_X\phi(x)\pi(\mathrm{d}x)\right\|_\beta.
     \end{equation*}In particular, for any $\alpha_0\in (0,\alpha)$ and $\gamma_0\in (\gamma+\frac{2C}{R},1)$, one can choose $\bar{\alpha}=(1+\alpha_0-\alpha)\vee \frac{2+R\beta\gamma_0}{2+R\beta}$ and $\beta=\frac{\alpha_0}{C}$.
\end{theorem}

\subsection{Random dynamical systems and Lyapunov exponents}
The Markov processes considered in the previous section arise from \textit{random dynamical systems} (RDS) with independent increments. In this section, we review several basic notions from the theory of RDS that will be needed in our analysis. 

Let $\left(\Omega_0,\mathcal{F}_0,\mathbb{P}_0\right)$ be a fixed probability space. For each $\omega\in\Omega_0$, we associate a measurable map $f_\omega\colon X\rightarrow X$. In analogy with the notation introduced in Section~\ref{sec: introduction}, we consider the product probability space \[\left(\Omega,\mathcal{F},\mathbb{P}\right):=\left(\Omega_0,\mathcal{F}_0,\mathbb{P}_0\right)^\mathbb{N}\]
whose typical element is $\underline{\omega}=(\omega_1,\omega_2,\ldots)\in\Omega$. For such a sequence $\underline{\omega}$ we define the corresponding random composition by 
\begin{equation*}
    f_{\underline{\omega}}^n:=f_{\omega_n}\circ\cdots\circ f_{\omega_1}.
\end{equation*}Moreover, let $\theta\colon\Omega\rightarrow\Omega$ denote the left shift map, defined by $\theta\underline{\omega}=(\omega_2,\omega_3,\ldots)$. It is well known that $\left(\Omega,\mathcal{F},\mathbb{P},\theta\right)$ is a measure-preserving system. We now recall the formal definition of a continuous RDS.
\begin{definition}
    A continuous RDS on $X$ over $\left(\Omega,\mathcal{F},\mathbb{P},\theta)\right)$ is a mapping
    \begin{equation*}
        f\colon \mathbb{N}\times \Omega \times X \to X, \qquad (n,\underline{\omega},x)\mapsto f_{\underline{\omega}}^n(x),
    \end{equation*}satisfying the following properties:
    \begin{itemize}
        \item[(1)] (Measurability) For all $x\in X$ and $A\in\mathcal{B}(X)$, the set $\left\{\omega\in\Omega_0: f_\omega(x)\in A\right\}$ is $\mathcal{F}_0$-measurable.
        \item[(2)] (Continuity) The mapping $f_\omega\colon X\rightarrow X$ is continuous for all $\omega\in\Omega_0$.
        \item[(3)] (Identity) $f_{\underline{\omega}}^0=\operatorname{Id}$ for all $\underline{\omega}\in\Omega$.
        \item[(4)] (Cocycle property) For all $m,n\in\mathbb{N}$ and all $\underline{\omega}\in\Omega$, we have that 
        \begin{equation*}
            f_{\underline{\omega}}^{n+m}=f_{\theta^m\underline{\omega}}^n\circ f_{\underline{\omega}}^m.
        \end{equation*}
    \end{itemize}
\end{definition}In analogy with \eqref{eq:one-point}, we define the Markov transition kernel associated with the continuous RDS $f_{\underline{\omega}}^n$ by
\begin{equation*}
    P(x,A)=\mathbb{P}_0(f_\omega(x)\in A).
\end{equation*}The continuity of $f_{\omega}$ immediately implies that $P$ enjoys the \textit{Feller property}. In addition, we say that a probability measure $\pi\in\mathcal{P}(X)$ is a \textit{stationary} (respectively, \textit{ergodic stationary}) measure for the continuous RDS $f_{\underline{\omega}}^{\,n}$ if $\pi$ is stationary (respectively, ergodic stationary) for the corresponding Markov transition kernel $P$.

\vspace{0.3em}
When the maps $f_\omega$ are smooth, their derivatives define a measurable mapping \begin{equation*}
    D_xf_\omega\colon \Omega_0\times X\rightarrow GL_d(\mathbb{R}),
\end{equation*} where $d=\dim X$. The linear cocycle generated by $D_xf_\omega$ is the composition 
\begin{equation*}
     D_xf_{\underline{\omega}}^n:=D_{f_{\underline{\omega}}^{n-1}(x)}f_{\omega_n}\circ\cdots \circ D_xf_{\omega_1},
\end{equation*}which describes the evolution of tangent vectors along the trajectory starting from $x$. 

\section{Quantitative exponential mixing for incompressible RDS}\label{sec: Quantitative exponential mixing for incompressible RDS}
Consider a family of continuous random dynamical systems $\left\{f_\omega^A\right\}_{A\in\mathbb{R}^+}$ on the state space $X=\mathbb{T}^d$, parameterized by $A$. For notational convenience, we write $f_\omega$ in place of $f_\omega^A$ whenever the dependence on $A$ is unambiguous. The primary objective of this section is to provide a systematic approach to verifying quantitative exponential mixing for $f_\omega$. To this end, we work under some basic standing assumptions.
\begin{enumerate}[label=($\mathbf{H}$), ref=$\mathbf{H}$]
    \item\label{hypothese H_0}
    The continuous RDS $f_\omega$ satisfies the following:
    \begin{enumerate}[label=(\roman*), ref=\textup{\roman*}]
        \item\label{hypothese H_1}
        Let $(\Omega_0,\mathcal{F}_0,\mathbb{P}_0)$ be the noise space, where $\Omega_0=\mathbb{R}^d$. The probability measure $\mathbb{P}_0$ admits a density $\rho_0$ with respect to Lebesgue measure $\mathrm{d}\omega$ on $\Omega_0$, that is,
        \begin{equation*}
            \mathrm{d}\mathbb{P}_0(\omega)=\rho_0(\omega)\,\mathrm{d}\omega.
        \end{equation*}
        Moreover, the mapping $(\omega,x)\mapsto f_\omega(x)$ is of class $C^2$ on $\Omega_0\times X$.

        \item\label{hypothese H_2}
        For any $A\in\R^+$, there exists a constant $C_0=C_0(A)>0$ such that for $\mathbb{P}_0$-a.e.\ $\omega$,
        \begin{equation*}
            \|D_x f_\omega\|\leq C_0,
            \qquad
            \|(D_x f_\omega)^{-1}\|\leq C_0,
            \qquad
            \|f_\omega\|_{C^2}\leq C_0.
        \end{equation*}

        \item\label{hypothese H_3}
        For $\mathbb{P}_0$-a.e.\ $\omega$, the map $f_\omega$ preserves $\pi=\operatorname{Leb}$ on $X$. Consequently, $\pi$ is a stationary measure for the associated Markov process.
    \end{enumerate}
\end{enumerate}
We remark that these structural conditions are natural and are, for instance, satisfied by the randomized Chirikov standard map. Building upon these hypotheses, we now proceed to establish a sufficient condition to rigorously quantify the mixing rate of $f_\omega$.
\begin{theorem}\label{thm: general framework for quantitative exponential mixing}
    Assume that, in addition to hypothesis \textup{(\ref{hypothese H_0})}, the following conditions hold:
    \begin{enumerate}[label=(\arabic*), ref=(\arabic*)]
        \item \label{assumption: general uniform contraction}There exist $A$-independent constants $m\in\mathbb{N}$, $p\in (0,1)$ and $\gamma'\in \left(0,\frac{1}{2}\right)$ such that for any $x\in \mathbb{T}^d$ and $v\in\mathbb{S}^{d-1}$, we have 
        \begin{equation}
            \mathbb{E}\left\|D_xf_{\underline{\omega}^m}(x)v\right\|^{-p}\leq \gamma',\label{eq: general uniform contraction}
        \end{equation}where the expectation is taken with respect to $\underline{\omega}^m\in \Omega_0^m$.
        
        \item \label{assumption: general derivative control}For any $A\in\mathbb{R}^+$, there exists a constant $C_1=C_1(A)$ such that for all $n\leq m$ and $\mathbb{P}_0^{\otimes n}$-a.e. $\underline{\omega}^n$,
        \begin{equation}
            \left\|D_xf_{\underline{\omega}^n}\right\|\leq C_1,\qquad \left\|D_x^2f_{\underline{\omega}^n}\right\|\leq C_1.\label{eq: general derivative control}
        \end{equation}Furthermore, we assume that $C_1(A)\rightarrow +\infty$ as $A\rightarrow+\infty$.
        
        \item \label{assumption: general small set condition}For any $s>0$, there exist an integer $k=k(A,s)$, a probability measure $\nu$ and a constant $C_2(A,s)\in (0,1)$ such that, for all $x,y\in\mathbb{T}^d$ with 
        \begin{equation*}
            \operatorname{dist}_{\mathbb{T}^d}(x,y)\geq s,
        \end{equation*}we have 
        \begin{equation}
            P^{(2),mk}\!\left((x,y),\cdot\right)\geq C_2(A,s)\nu(\cdot),\label{eq: general small set condition}
        \end{equation}where $P^{(2),n}$ denotes the $n$-th step transition kernel for the two-point process \eqref{eq: def of two-point process} generated by $f_\omega$.
    \end{enumerate}Then, there exist constants $A_0,C',C'',s'>0$ such that for any $q>0$ and $A\geq A_0$, one can find a random variable $\hat{D}\colon\Omega\rightarrow [0,\infty)$ with the following properties.
    \begin{enumerate}[label=(\arabic*), ref=(\arabic*)]
        \item For all mean-zero observables $\varphi,\psi\in H^1(\mathbb{T}^d)$, we have
        \begin{equation*}
            \left|\int_{\mathbb{T}^d}\varphi(x)\psi\circ f_{\underline{\omega}}^n(x)\mathrm{d}x\right|\leq \hat{D}(\underline{\omega})e^{-\frac{C'}{1+q}\tau^2 n}\|\varphi\|_{\dot{H}^1}\|\psi\|_{\dot{H}^1}.
        \end{equation*}Here $C'=C'(d,p)>0$ depends only on $d$ and $p$, and
        \begin{equation*}
            \tau:=\min\!\left\{C_2(A,C_1^{-(2d-1)}s'),\left(mk(A,C_1^{-(2d-1)}s')\right)^{-1}\right\}.
        \end{equation*}
        
        \item The random variable $\hat{D}$ satisfies the moment bound
        \begin{equation*}
            \mathbb{E}[\hat{D}^q]\leq C''2^{k(A,C_1^{-(2d-1)}s')/2}C_1^{(d-1)p/2}.
        \end{equation*}
    \end{enumerate}
\end{theorem}
\begin{rmk}
    We mention that once the explicit dependence of $C_1(A),C_2(A,s)$ and $k(A,s)$ on the parameter $A,s$ is established, the quantitative mixing rate follows immediately as a direct consequence of Theorem \ref{thm: general framework for quantitative exponential mixing}.
\end{rmk}
\begin{rmk}
    In addition to the randomized Chirikov standard map analyzed in this paper, Theorem \ref{thm: general framework for quantitative exponential mixing} can also be applied to the Pierrehumbert model considered in \cite{BCZG-23,CIS-25,son-25}. We defer the detailed application to Appendix \ref{sec: Application to the Pierrehumbert model}. Furthermore, we expect that our methodology can be effectively adapted to study the ABC flow investigated in \cite{coti-26}.
\end{rmk}
\begin{rmk}
    In Assumption \ref{assumption: general uniform contraction}, we have restricted the parameter ranges to $p\in (0,1)$ and $\gamma'\in \left(0,\frac{1}{2}\right)$, primarily to simplify the subsequent estimates. In fact, to obtain quantitative exponential mixing, it suffices to relax the condition to $\gamma'\in (0,1)$ and $p>0$, though this would yield a different explicit mixing rate than the one presented in Theorem \ref{thm: general framework for quantitative exponential mixing}.
\end{rmk}
\begin{rmk}
    The derivative control in Assumption~\ref{assumption: general derivative control} can be weakened (e.g., to bounds for $n=1,2$ only), at the cost of modifying the constants appearing in the drift estimates.
\end{rmk}The proof of Theorem \ref{thm: general framework for quantitative exponential mixing} consists of three steps. First, by combining Assumptions \ref{assumption: general uniform contraction}, \ref{assumption: general derivative control}, we construct an explicit Lyapunov function $V(x,y)$ for the two-point process. This construction is crucial for establishing a quantitative drift condition. Second, setting $M=mk$, we utilize the small set condition from Assumption \ref{assumption: general small set condition} alongside our drift inequality to verify the hypotheses of the quantitative Harris theorem (Theorem \ref{thm: quantitative harris}). This yields a contraction estimate for the iterated kernel $P^{(2),Mn}$, which we subsequently extend to all $n\in\mathbb{N}$. Finally, we conclude the proof by translating these probabilistic bounds on the two-point process into correlation decay for the original dynamics.

\subsection{Construction of Lyapunov function}\label{subsec: general construction of Lyapunov function}For any $x\neq y$, we can choose lifts $\tilde{x},\tilde{y}\in\mathbb{R}^d$ such that
    \begin{equation*}
        \|\tilde{y}-\tilde{x}\|=\operatorname{dist}_{\mathbb{T}^d}(x,y),
    \end{equation*}where $\|\cdot\|$ denotes the Euclidean norm in $\mathbb{R}^d$. Let $\tilde{f}_{\underline{\omega}^m}$ be a lift of $f_{\underline{\omega}^m}$ to $\mathbb R^d$. Then Taylor's theorem gives
    \begin{equation*}
        \tilde{f}_{\underline{\omega}^m}(\tilde{y})-\tilde{f}_{\underline{\omega}^m}(\tilde{x})=D_{\tilde{x}}\tilde{f}_{\underline{\omega}^m}(\tilde{x})\cdot (\tilde{y}-\tilde{x})+R_2(\tilde{x},\tilde{y}),
    \end{equation*}where the remainder term satisfies $\left\|R_2(\tilde{x},\tilde{y})\right\|=O(\|\tilde{y}-\tilde{x}\|^2)$. Leveraging the Lipschitz estimates \eqref{eq: general derivative control} and the incompressibility of $f_\omega$, we can bound the linear and remainder terms as follows:
    \begin{equation*}
        \left\|D_{\tilde{x}}\tilde{f}_{\underline{\omega}^m}(\tilde{x})\cdot (\tilde{y}-\tilde{x})\right\|\geq \frac{1}{C_1^{d-1}}\|\tilde{y}-\tilde{x}\|,\qquad \left\|R_2(\tilde{x},\tilde{y})\right\|\leq C_1\|\tilde{y}-\tilde{x}\|^2.
    \end{equation*}By the triangle inequality, we obtain that
    \begin{align*}
        \left\|\tilde{f}_{\underline{\omega}^m}(\tilde{x})-\tilde{f}_{\underline{\omega}^m}(\tilde{y})\right\|&\geq \left\|D_{\tilde{x}}\tilde{f}_{\underline{\omega}^m}(\tilde{x})\cdot (\tilde{y}-\tilde{x})\right\|-\left\|R_2(\tilde{x},\tilde{y})\right\|\\[0.2em]
        &\geq \left\|D_{\tilde{x}}\tilde{f}_{\underline{\omega}^m}(\tilde{x})\cdot (\tilde{y}-\tilde{x})\right\|\!\left(1-\frac{C_1(A)\|\tilde{y}-\tilde{x}\|^2}{\left\|D_{\tilde{x}}\tilde{f}_{\underline{\omega}^m}(\tilde{x})\cdot (\tilde{y}-\tilde{x})\right\|}\right)\\[0.2em]
        &\geq \left\|D_{\tilde{x}}\tilde{f}_{\underline{\omega}^m}(\tilde{x})\cdot (\tilde{y}-\tilde{x})\right\|\!\left(1-C_1^d\|\tilde{y}-\tilde{x}\|\right).
    \end{align*}Consequently, providing that we restrict to the region 
    \begin{equation*}
        \Delta\!\left(\!\left(2C_1^d\right)^{-1}\right):=\left\{(x,y):0<\operatorname{dist}_{\mathbb{T}^d}(x,y)<\frac{1}{2C_1^d}\right\},
    \end{equation*}the following lower bound holds
    \begin{equation}\label{eq: general lower bound for f_omega(x)-f_omega(y)}
        \left\|\tilde{f}_{\underline{\omega}^m}(\tilde{x})-\tilde{f}_{\underline{\omega}^m}(\tilde{y})\right\|\geq \frac{1}{2}\left\|D_{\tilde{x}}\tilde{f}_{\underline{\omega}^m}(\tilde{x})\cdot (\tilde{y}-\tilde{x})\right\|
    \end{equation}Since $\frac{\tilde{y}-\tilde{x}}{\|\tilde{y}-\tilde{x}\|}=\frac{\tilde{y}-\tilde{x}}{\operatorname{dist}_{\mathbb{T}^d}(x,y)}\in\mathbb{S}^{d-1}$, applying Assumption \ref{assumption: general uniform contraction} yields
    \begin{equation}\label{eq: general contraction for tilde(x)-tilde(y)}
        \mathbb{E}\!\left\|D_{\tilde{x}}\tilde{f}_{\underline{\omega}^m}\cdot (\tilde{y}-\tilde{x})\right\|^{-p}=\|\tilde{y}-\tilde{x}\|^{-p}\cdot\mathbb{E}\!\left\|D_xf_{\underline{\omega}^m}(x)\cdot\frac{\tilde{y}-\tilde{x}}{\|\tilde{y}-\tilde{x}\|}\right\|^{-p}\leq \gamma'\operatorname{dist}_{\mathbb{T}^d}(x,y)^{-p}.
    \end{equation}
        
    Define the Lyapunov function 
    \begin{equation*}
        V\colon\left(\mathbb{T}^d\times\mathbb{T}^d\right)\setminus\Delta\rightarrow[0,+\infty),\qquad V(x,y)=\operatorname{dist}_{\mathbb{T}^d}(x,y)^{-p}.
    \end{equation*}Then for any $(x,y)\in \Delta\!\left(\!\left(2C_1^d\right)^{-1}\right)$, assuming $A>A_0$ so that $\left\|\tilde{f}_{\underline{\omega}^m}(\tilde{y})-\tilde{f}_{\underline{\omega}^m}(\tilde{x})\right\|\leq\frac{1}{2C_1^{d-1}}$ is sufficiently small, we have 
    \begin{equation}\label{eq: f_omega(x)-f_omega(y)=tilde(f)(x)-tilde(f)(y)}
        \operatorname{dist}_{\mathbb{T}^d}\!\left(f_{\underline{\omega}^m}(x),f_{\underline{\omega}^m}(y)\right)=\left\|\tilde{f}_{\underline{\omega}^m}(\tilde{y})-\tilde{f}_{\underline{\omega}^m}(\tilde{x})\right\|.
    \end{equation}Combining \eqref{eq: general lower bound for f_omega(x)-f_omega(y)}, \eqref{eq: general contraction for tilde(x)-tilde(y)} and \eqref{eq: f_omega(x)-f_omega(y)=tilde(f)(x)-tilde(f)(y)}, we derive the drift condition near the diagonal
    \begin{align}
        P^{(2),m}V(x,y)&=\mathbb{E}\!\left[\operatorname{dist}_{\mathbb{T}^d}\!\left(f_{\underline{\omega}^m}(x),f_{\underline{\omega}^m}(y)\right)^{-p}\right]\notag\\[0.2em]
        &\leq2^p\mathbb{E}\left\|D_{\tilde{x}}\tilde{f}_{\underline{\omega}^m}(\tilde{x})\cdot(\tilde{y}-\tilde{x})\right\|^{-p}\notag\\[0.2em]
        &\leq 2^p\gamma'\operatorname{dist}_{\mathbb{T}^d}(x,y)^{-p}\notag\\[0.2em]
        &=2^p\gamma'V(x,y)=:\gamma V(x,y),\label{eq: general drift condition near diagonal}
    \end{align}where $\gamma=2^p\gamma'\leq 2^{p-1}\in(0,1)$.

   Away from the diagonal, for $(x,y)\in \Delta\!\left(\!\left(2C_1^d\right)^{-1}\right)^c$, the crude lower bound
   \begin{equation*}
        \operatorname{dist}_{\mathbb{T}^d}\!\left(f_{\underline{\omega}^m}(x),f_{\underline{\omega}^m}(y)\right)\geq \frac{\operatorname{dist}_{\mathbb{T}^d}(x,y)}{C_1^{d-1}}
    \end{equation*}yields that
    \begin{equation}\label{eq: general drift condition away diagonal}
        P^{(2),m}V(x,y)\leq C_1^{(d-1)p}V(x,y)\leq 2^pC_1^{(2d-1)p}.
    \end{equation}
    Combining the two regimes \eqref{eq: general drift condition near diagonal} and \eqref{eq: general drift condition away diagonal}, we obtain the global drift condition
    \begin{equation}\label{eq: general quantitative drift condition}
        P^{(2),m}V\leq \gamma V+C_3(A),
    \end{equation}
    where $C_3=C_3(A):=2^pC_1^{(2d-1)p}$.
    
\subsection{Contraction estimate for the transition kernel}\label{subsec: Contraction Estimate for the Transition Kernel}
This section is devoted to establishing a crucial exponential contraction estimate for the two-point transition kernel. Building upon Assumption \ref{assumption: general small set condition} and the Lyapunov function defined previously, we summarize this core result in the following proposition, which serves as the cornerstone for proving our main quantitative mixing result.
\begin{proposition}\label{prop: Contraction estimate for the transition kernel}
    Retaining the notation for $C_1,C_2,M,k$ and $p$ from the hypotheses of {\rm Theorem~\ref{thm: general framework for quantitative exponential mixing}}, there exist constants $A_0,s',l_0>0$ such that for all $A\geq A_0$ and any mean-zero observable $\varphi\in H^1\!\left(\!\left(\mathbb{T}^d\times\mathbb{T}^d\right)\setminus\Delta\right)$, the following bound holds:
    \begin{equation*}
        \left\|P^{(2),n}\varphi\right\|_\beta\leq 2^{k+1}C_1^{(d-1)p}e^{-\frac{1}{l_0}\tau^2n}\|\varphi\|_\beta,
    \end{equation*}where $\|\cdot\|_\beta$ denotes the weighted norm defined in \eqref{eq: beta-norm}, and the contraction rate $\tau$ is explicitly given by
    \begin{equation*}
        \tau:=\min\!\left\{C_2\!\left(A,C_1^{-(2d-1)}s'\right),M^{-1}\right\}.
    \end{equation*}
\end{proposition}
\begin{proof}
    Iterating the $m$-step drift inequality \eqref{eq: general quantitative drift condition} yields, for every $n\in\mathbb{N}$,
    \begin{equation*}
        P^{(2),mn}V\leq \gamma^nV+C_3\frac{1-\gamma^n}{1-\gamma}.
    \end{equation*}Define the constant $R_2:=\frac{4C_3}{1-\gamma}$. Recalling that $V(x,y)=\operatorname{dist}_{\mathbb{T}^d}(x,y)^{-p}$, the condition $V(x,y)\leq R_2$ implies
    \begin{equation*}
        \operatorname{dist}_{\mathbb{T}^d}(x,y)\geq R_2^{-1/p}=\frac{1}{2}\left(\frac{1-\gamma}{4}\right)^{1/p}C_1^{-(2d-1)}.
    \end{equation*}Hence, there exists an $A$-independent constant $s'>0$ such that
    \begin{equation*}
        \{V\leq R_2\}\subset \Delta\!\left(C_1^{-(2d-1)}s'\right)^c.
    \end{equation*}Applying \eqref{eq: general small set condition} to the region $\Delta\!\left(C_1^{-(2d-1)}s'\right)^c$, we can select an integer $k$ as in Assumption \ref{assumption: general small set condition} such that, setting $M=mk$
    \begin{equation}\label{eq: general quantitative small set condition for P^M}
        \inf_{(x,y)\in \Delta\!\left(C_1^{-(2d-1)}s'\right)^c}P^{(2),M}\!\left((x,y),\cdot\right)\geq \alpha\nu(\cdot),
    \end{equation}where $k=k\!\left(A,C_1^{-(2d-1)}s'\right)$ and $\alpha=\alpha(A):=C_2\!\left(A,C_1^{-(2d-1)}s'\right)$ depend only on the parameter $A$. Simultaneously, substituting $n=k$ into our iterated drift bound gives
    \begin{equation}\label{eq: general quantitative drift condition for P^M}
        P^{(2),M}V\leq \gamma^kV+C_3\frac{1-\gamma^k}{1-\gamma}.
    \end{equation}

    We now introduce the following auxiliary constants
    \begin{align*}
        &\qquad\qquad\qquad L_{mn}:=C_3\frac{1-\gamma^n}{1-\gamma},\quad  n\geq 1,\\
        &\gamma_0:=\frac{3+\gamma^k}{4},\quad 
        \beta:=\frac{\alpha}{2L_M},\quad 
        \bar{\alpha}:=\max\!\left\{1-\frac{\alpha}{2},\frac{2+\beta R_2\gamma_0}{2+\beta R_2}\right\}.
    \end{align*}Observing that $\frac{L_M}{R_2}=\frac{1-\gamma^k}{4}$, we have
    \begin{equation*}
        \gamma_0=\frac{3+\gamma^k}{4}\in\left(\frac{1+\gamma^k}{2},1\right)=\left(\gamma^k+\frac{2L_M}{R_2},1\right).
    \end{equation*}Combining \eqref{eq: general quantitative small set condition for P^M} and \eqref{eq: general quantitative drift condition for P^M} with notations precisely defined, we verify that $P^{(2),M}$ satisfies the hypotheses of the quantitative Harris theorem (Theorem \ref{thm: quantitative harris}). Consequently, for any mean-zero function $\varphi\in H^{1}$, we obtain the contraction estimate for the iterated kernel 
    \begin{equation}\label{eq: general contraction on P^Mn}
        \left\|P^{(2),Mn}\varphi\right\|_\beta\leq \bar{\alpha}^n\|\varphi\|_\beta,
    \end{equation}where $\|\cdot\|_\beta$ is the weighted norm defined in \eqref{eq: beta-norm}.

    To extend this estimate to any time step $n\in\mathbb{N}$, we first establish an $r$-step bound for $r\leq m-1$. The Lipschitz estimate \eqref{eq: general derivative control} implies the lower bound
    \begin{equation*}
        \operatorname{dist}_{\mathbb{T}^d}\!\left(f_{\underline{\omega}^r}(x),f_{\underline{\omega}^r}(y)\right)\geq \frac{\operatorname{dist}_{\mathbb{T}^d}(x,y)}{C_1(A)^{d-1}},
    \end{equation*}and hence 
    \begin{equation*}
        P^{(2),r}V(x,y)=\mathbb{E}\!\left[\operatorname{dist}_{\mathbb{T}^d}\!\left(f_{\underline{\omega}^r}(x),f_{\underline{\omega}^r}(y)\right)^{-p}\right]\leq C_1^{(d-1)p}V(x,y).
    \end{equation*}Consequently, for any $(x,y)\in \left(\mathbb{T}^d\times\mathbb{T}^d\right)\setminus\Delta$, we deduce that
    \begin{align*}
        \left|\frac{P^{(2),r}\varphi(x,y)}{1+\beta V(x,y)}\right|&\leq \int\!\left|\frac{\varphi(x',y')}{1+\beta V(x',y')}\frac{1+\beta V(x',y')}{1+\beta V(x,y)}\right|P^{(2),r}\!\left((x,y),\mathrm{d}(x',y')\right)\\[0.2em]
        &\leq \|\varphi\|_\beta\ \frac{1+\beta P^{(2),r}V(x,y)}{1+\beta V(x,y)}\\[0.2em]
        &\leq \|\varphi\|_\beta\ \frac{1+\beta C_1^{(d-1)p}V(x,y)}{1+\beta V(x,y)}\leq \|\varphi\|_\beta\!\left(1+C_1^{(d-1)p}\right).
    \end{align*}Furthermore, applying the drift condition \eqref{eq: general quantitative drift condition for P^M} yields 
    \begin{align*}
        \left|\frac{P^{(2),m}\varphi(x,y)}{1+\beta V(x,y)}\right|&\leq \int\!\left|\frac{\varphi(x',y')}{1+\beta V(x',y')}\frac{1+\beta V(x',y')}{1+\beta V(x,y)}\right|P^{(2),m}\!\left((x,y),\mathrm{d}(x',y')\right)\\[0.2em]
        &\leq \|\varphi\|_\beta\ \frac{1+\beta(\gamma V(x,y)+L_m)}{1+\beta V(x,y)}\\[0.2em]
        &\leq \|\varphi\|_\beta(1+\beta L_m)\leq \|\varphi\|_\beta (1+\alpha).
        \end{align*}Thus we obtain
    \begin{equation*}
        \left\|P^{(2),r}\varphi\right\|_\beta\leq \left(1+C_1^{(d-1)p}\right)\!\|\varphi\|_\beta,\qquad \left\|P^{(2),m}\varphi\right\|_\beta\leq (1+\alpha)\|\varphi\|_\beta.
    \end{equation*}
        
    For an arbitrary $n\in\mathbb{N}$, we decompose it as $n=Ml_1+ml_2+r$, where $0\leq l_2\leq k-1$ and $0\leq r\leq m-1$. Combining the contraction \eqref{eq: general contraction on P^Mn} of $P^{(2),lM}$ and the bound for the remainder steps, we have
    \begin{equation}\label{eq: general bound for P^n with bar-alpha}
        \left\|P^{(2),n}\varphi\right\|_\beta\leq \left(1+C_1^{(d-1)p}\right)(1+\alpha)^k\bar{\alpha}^{\frac{n}{M}-1}\|\varphi\|_\beta.
    \end{equation}For $A$ sufficiently large, $C_1(A)$ is large enough to ensure $1+C_1^{(d-1)p}\leq 2C_1^{(d-1)p}$. Observing that
    \begin{equation*}
        \frac{2+\beta R_2\gamma_0}{2+\beta R_2}=1-\frac{1}{4}\cdot\frac{\alpha}{1+(1-\gamma^k)^{-1}\alpha}>1-\frac{\alpha}{2},
    \end{equation*}we deduce that $\bar{\alpha}=\frac{2+\beta R_2\gamma_0}{2+\beta R_2}$. Moreover, since $\alpha\in (0,1)$ and $\gamma<2^{p-1}$, it follows that
    \begin{equation*}
        \frac{2+\beta R_2\gamma_0}{2+\beta R_2}=1-\frac{1}{4}\cdot\frac{\alpha}{1+(1-\gamma^k)^{-1}\alpha}<1-\frac{\alpha}{l_0},
    \end{equation*}where $l_0=4\left(1+\frac{1}{1-2^{p-1}}\right)$. Substituting these bounds into \eqref{eq: general bound for P^n with bar-alpha}, we conclude
    \begin{equation*}
        \left\|P^{(2),n}\varphi\right\|_\beta\leq 2^{k+1}C_1^{(d-1)p}\left(1-\frac{\alpha}{l_0}\right)^{\frac{n}{M}}\|\varphi\|_\beta\leq 2^{k+1}C_1^{(d-1)p}e^{-\frac{\alpha}{l_0M}n}\|\varphi\|_\beta.
    \end{equation*}Letting $\tau:=\min\{\alpha,M^{-1}\}$, we obtain the desired estimate
    \begin{equation}\label{eq: general quantitative drift condition for P^n}
        \left\|P^{(2),n}\varphi\right\|_\beta\leq 2^{k+1}C_1^{(d-1)p}e^{-\frac{1}{l_0}\tau^2n}\|\varphi\|_\beta.
    \end{equation}
\end{proof}
Having established the contraction estimates, completing the proof of Theorem \ref{thm: general framework for quantitative exponential mixing} is straightforward. As the concluding arguments closely follow standard strategy introduced in \cite{DKK-04} and refined in \cite{BBPS-22A,BCZG-23,CIS-25}, we defer the full technical details to Appendix 
\ref{subsec: Proof of Theorem 1.1}.

\section{Quantitative mixing rate: proof of Theorem \ref{thm:quantitative}}\label{sec: Quantitative Mixing Rate: Proof of Theorem 2}
We now return to the randomized Chirikov standard map presented in Section \ref{sec: introduction}. This section is devoted to the proof of Theorem \ref{thm:quantitative}, which provides a quantitative estimate for the mixing rate. Our strategy is to apply the method established in Section \ref{sec: Quantitative exponential mixing for incompressible RDS}. Accordingly, in Section \ref{subsec: Uniform contraction condition} and \ref{subsec: Quantitative small set condition}, we verify the uniform contraction condition (Assumption \ref{assumption: general uniform contraction}) and the quantitative small set condition (Assumption \ref{assumption: general small set condition}), respectively. These elements are finally combined to derive the desired quantitative exponential mixing. We begin by introducing some notation and lemmas.

\subsection{A Priori bounds}
Throughout this section,  we fix the following notation for the diagonal
\begin{equation*}
    \Delta:=\{(x,x):x\in\mathbb{T}^2\}\subset\mathbb{T}^2\times\mathbb{T}^2,
\end{equation*}and denote points in $\left(\mathbb{T}^2\times\mathbb{T}^2\right)\setminus\Delta$ by $z=(x,y)$. For a prescribed sequence of phases $\underline{\omega}^n=(\omega_1,\ldots,\omega_n)\in \Omega_0^n$, we define the $n$-step two-point map $\Phi_{2n}\colon \left(\left(\mathbb{T}^2\times\mathbb{T}^2\right)\setminus \Delta\right)\times\Omega_0^n\rightarrow \left(\mathbb{T}^2\times\mathbb{T}^2\right)\setminus \Delta$ as
\begin{equation*}
    \Phi_{2n}\left((z,\underline{\omega}^n\right):=\left(f_{\omega_n}\circ\cdots\circ f_{\omega_1}(x),f_{\omega_n}\circ\cdots\circ f_{\omega_1}(y)\right),
\end{equation*}Accordingly, the $n$-step transition kernel for the two-point process is given by
\begin{equation*}
    P^{(2),n}\left(z,A\right):=\mathbb{P}\left(\Phi_{2n}\!\left(z,\underline{\omega}^n\right)\in A\right),
\end{equation*}for any Borel set $A\in\mathcal{B}\left(\left(\mathbb{T}^2\times\mathbb{T}^2\right)\setminus\Delta\right)$. The following lemma summarizes a priori estimates for $\Phi_{2n}$. 
\begin{lemma}\label{lem: priori estimates}
    For any $n\in\mathbb{N}$, there exists some constant $C=C(n)$ such that, for all $z\in \left(\mathbb{T}^2\times\mathbb{T}^2\right)\setminus \Delta$ and $\underline{\omega}^n\in\Omega_0^n$, the following estimates hold
    \begin{align}
        \left\|D_{\omega_i}\Phi_{2n}\left(z,\underline{\omega}^n\right)\right\|&\leq CK^n,\label{eq:first derivative omega_i}\\
        \left\|D_{z_j}\Phi_{2n}\left(z,\underline{\omega}^n\right)\right\|&\leq CK^n,\label{eq:first derivative z_j}\\
        \left\|D_{\omega_i}D_{\omega_{i'}}\Phi_{2n}\left(z,\underline{\omega}^n\right)\right\|&\leq CK^{2n+1},\label{eq:second derivative omega_i,i'}\\
        \left\|D_{z_j}D_{\omega_{i}}\Phi_{2n}\left(z,\underline{\omega}^n\right)\right\|&\leq CK^{2n+1},\label{eq:second derivative omega_i z_j}\\
        \left\|D_{z_j}D_{z_{j'}}\Phi_{2n}\left(z,\underline{\omega}^n\right)\right\|&\leq CK^{2n+1},\label{eq:second derivative z_j,j'}
    \end{align}
\end{lemma}
\noindent The proof of Lemma \ref{lem: priori estimates} follows by a direct induction on $n$ and is omitted here for brevity. 

\subsection{Uniform contraction condition}\label{subsec: Uniform contraction condition}

Equipped with these auxiliary results, the next lemma demonstrates that the uniform contraction condition is satisfied for $m=2$.
\begin{lemma}\label{lem: uniform derivative drift condition}
   There exist $K$-independent constants $p,\gamma'\in\!\left(0,\frac{1}{2}\right)$ such that for any $x\in\mathbb{T}^2$ and $v\in\mathbb{S}^1$, we have 
   \begin{equation}\label{eq: derivative drift condition}
       \mathbb{E}\left\|D_xf_{\underline{\omega}^2}(x)v\right\|^{-p}\leq \gamma',
   \end{equation}where the expectation is taken with respect to $\underline{\omega}^2\in \Omega_0^2$.
\end{lemma}
\begin{proof}
    For any $x=(x_1,x_2)\in\mathbb{T}^2$ and $\omega=(\omega^1,\omega^2)\in\Omega_0$, the one-step Jacobian matrix $D_xf_\omega$ is given by
    \begin{equation*}
        D_xf_{\omega}=\begin{pmatrix}
            1 & K\cos\!\left(x_2-\omega^1\right)\\
            1 & 1+K\cos\!\left(x_2-\omega^1\right)
        \end{pmatrix}=:J(\theta),
    \end{equation*}where $\theta:=x_2-\omega^1$ is uniformly distributed on $[0,2\pi)$. For notational convenience, we define the iterated tangent vectors
    \begin{equation*}
        w_1:=D_xf_{\omega_1}(x)v=J(\theta_1)v,\qquad w_2:=D_xf_{\underline{\omega}^2}(x)v=J(\theta_2)J(\theta_1)v=(w_{2,1},w_{2,2}),
    \end{equation*}where $\theta_1,\theta_2\sim U\!\left([0,2\pi)\right)$ are independent. Since $\|w_2\|\geq |w_{2,2}|$, to establish \eqref{eq: derivative drift condition}, it suffices to show that
     \begin{equation*}
        \mathbb{E}|w_{2,2}|^{-p}\leq \gamma,
    \end{equation*}for some $\gamma'\in\!\left(0,\frac{1}{2}\right)$ (independent of $K$).

    Substituting the explicit expression
    \begin{equation*}
        w_{2,2}=w_{1,1}+w_{1,2}+Kw_{1,2}\cos(\theta_2)
    \end{equation*}into the expectation formula, we obtain the integral representation
    \begin{equation}\label{eq: integral representation of w_2,2}
        \mathbb{E}|w_{2,2}|^{-p}=\frac{1}{4\pi^2}\int_0^{2\pi}\!\int_0^{2\pi}\!\left|w_{1,1}+w_{1,2}+Kw_{1,2}\cos(\theta_2)\right|^{-p}\mathrm{d}\theta_2\mathrm{d}\theta_1.
    \end{equation}Our subsequent estimates rely on the following key claim.

    \begin{claim}\label{claim: control of a+b cos-theta}
        Let $a,b\in\mathbb{R}$ with $b\neq 0$. For any $p\in (0,\frac{1}{2})$, there exists a constant $C_p$, independent of $a$ and $b$, such that
        \begin{equation*}
            \int_{0}^{2\pi}|a+b\cos\theta|^{-p}\mathrm{d}\theta\leq C_p|b|^{-p}.
        \end{equation*}
    \end{claim}

    Applying Claim \ref{claim: control of a+b cos-theta} to the inner integral in \eqref{eq: integral representation of w_2,2} (with respect to $\theta_2$), we obtain
    \begin{equation}\label{eq: estimate of w_2,2 by w_1,2}
        \mathbb{E}|w_{2,2}|^{-p}\leq \frac{1}{4\pi^2}\int_0^{2\pi}C_p|Kw_{1,2}|^{-p}\mathrm{d}\theta_1=\frac{1}{4\pi^2}C_pK^{-p}\int_0^{2\pi}|w_{1,2}|^{-p}\mathrm{d}\theta_1.
    \end{equation}Note that the condition $w_{1,2}\neq 0$ holds for almost every $\theta_1\in [0,2\pi)$, allowing us to neglect the Lebesgue-null set where the singularity may occur. Although $w_{1,2}$ admits a representation analogous to that of $w_{2,2}$:
    \begin{equation*}
        w_{1,2}=v_1+v_2+Kv_2\cos(\theta_1),
    \end{equation*}we cannot apply Claim \ref{claim: control of a+b cos-theta} to \eqref{eq: estimate of w_2,2 by w_1,2} uniformly. Specifically, when $|v_2|$ is small, the resulting bound $C_p|Kv_2|^{-p}$ may become uncontrollably large, failing to provide the desired contraction. Therefore, we split the estimation into two regimes: $|Kv_2|\geq \frac{1}{2}$ and $|Kv_2|<\frac{1}{2}$.

     Assume that $|Kv_2|\geq \frac{1}{2}$. A direct application of Claim \ref{claim: control of a+b cos-theta} yields 
     \begin{equation*}
        \int_0^{2\pi}|w_{1,2}|^{-p}\mathrm{d}\theta_1=\int_0^{2\pi}\!\left|v_1+v_2+Kv_2\cos(\theta_1)\right|^{-p}\mathrm{d}\theta_1\leq C_p|Kv_2|^{-p}\leq C_p2^p.
    \end{equation*}Combining this with \eqref{eq: estimate of w_2,2 by w_1,2}, we have
    \begin{equation*}
        \mathbb{E}|w_{2,2}|^{-p}\leq \frac{1}{4\pi^2}2^pC_p^2K^{-p}.
    \end{equation*}Thus, there exists $K_1(p)\in\mathbb{N}$ such that for all $K\geq K_1$, the right-hand side is bounded by some $\gamma_1\in \left(0,\frac{1}{2}\right)$, independent of $K$.

    Assume that $|Kv_2|<\frac{1}{2}$. Since $v=(v_1,v_2)\in\mathbb{S}^1$, for sufficiently large $K$, we obtain 
    \begin{equation*}
        |w_{1,2}|\geq |v_1|-|v_2|-|Kv_2|\geq \sqrt{1-1/4K^2}-\frac{1}{2K}-\frac{1}{2}\geq\frac{1}{3}.
    \end{equation*}Substituting this pointwise bound into \eqref{eq: estimate of w_2,2 by w_1,2}, we deduce that
    \begin{equation*}
        \int_0^{2\pi}|w_{1,2}|^{-p}\mathrm{d}\theta_1\leq 2\pi\cdot 3^p.
    \end{equation*}Consequently, there exists $K_2(p)\in\mathbb{N}$ such that for all $K\geq K_2$,
     \begin{equation*}
        \mathbb{E}\|w_2\|^{-p}\leq \mathbb{E}|w_{2,2}|^{-p}\leq \frac{1}{2\pi}3^pC_pK^{-p}\leq \gamma_2,
    \end{equation*}for some $K$-independent constant $\gamma_2\in \left(0,\frac{1}{2}\right)$.

    Finally, fix any $p\in(0,\frac{1}{2})$ (e.g.\ $p=\frac{1}{4}$). Let $K_0=\max\{K_1,K_2\}$ and $\gamma'=\max\{\gamma_1,\gamma_2\}$. For any $K\geq K_0$, we conclude that
    \begin{equation*}
        \mathbb{E}\left\|D_xf_{\underline{\omega}^2}(x)v\right\|^{-p}=\mathbb{E}\|w_{2}\|^{-p}\leq\gamma'.
    \end{equation*}
\end{proof}
\begin{rmk}\label{rmk: gamma tends to 0}
   In particular, following the argument above, it is straightforward to show that as $K_0\rightarrow+\infty$, the drift coefficient $\gamma'$ can be chosen arbitrarily close to $0$.
\end{rmk}
Building upon the general construction outlined in Section \ref{subsec: general construction of Lyapunov function}, and utilizing estimates established in Lemma \ref{lem: priori estimates}, we obtain the following result as a direct consequence of Lemma \ref{lem: uniform derivative drift condition}.
\begin{proposition}\label{prop: quantitative drift condition}
    Provided that $K$ is sufficiently large, there exist constants $p>0$, $\gamma\in (0,1)$ and $C_0>0$, all independent of $K$, and a measurable function $V$
    \begin{equation*}
        V(x,y):=\operatorname{dist}_{\mathbb{T}^d}(x,y)^{-p},
    \end{equation*}satisfying the drift inequality:
    \begin{equation}\label{eq: quantitative drift condition}
        P^{(2),2}V\leq \gamma V+C_0K^{9p}.
    \end{equation}
\end{proposition}

\subsection{Quantitative small set condition}\label{subsec: Quantitative small set condition}
In this section, we verify Assumption \ref{assumption: general small set condition} of Theorem~\ref{thm: general framework for quantitative exponential mixing}. To achieve this, our proof is divided into three main steps. First, we establish a local quantitative small set condition at a reference point $(x_*,y_*)\in\left(\mathbb{T}^2\times\mathbb{T}^2\right)\setminus \Delta$. Second, we apply a controllability argument to prove the topological irreducibility of the two-point process. Third, utilizing this topological irreducibility, we demonstrate that the two-point process enters the small set of $(x_*,y_*)$ with a probability that admits an explicit lower bound in terms of $K$. Combining these ingredients, we deduce the desired global quantitative small set condition.

\subsubsection{Quantitative small set condition at $(x_*,y_*)$}
Throughout this section, we fix $z_*=(x_*,y_*):=\left((0,0),(0,\pi)\right)$. Straightforward calculations reveal that the Jacobian matrix
\begin{equation*}
    D_{\underline{\omega}^4}\Phi_8\!\left(z_*,\underline{\tilde{\omega}}^4\right)
\end{equation*}has full rank, where $\underline{\tilde{\omega}}^4=\boldsymbol{0}\in\Omega_0^4$. To make the notation precise, for $\underline{\omega}^4=\!\left(\omega_1^1,\omega_1^2,\ldots,\omega_4^2\right)$, we single out the four scalar noise coordinates
\begin{equation*}
    \xi=\!\left(\omega_1^1,\omega_2^2,\omega_3^2,\omega_4^2\right),
\end{equation*}and group the remaining components into
\begin{equation*}
    \theta=\!\left(\omega_1^2,\omega_2^1,\omega_3^1,\omega_4^1\right).
\end{equation*}Up to a canonical permutation of coordinates, we can naturally identify the space of $\underline{\omega}^4$ with the product space of $(\xi, \theta)$. Therefore, with a slight abuse of notation, we will write $\underline{\omega}^4=(\xi,\theta)$ in what follows to highlight this specific splitting. In particular, evaluating the partial Jacobian with respect to $\xi$ at $\tilde{\underline{\omega}}^4=\boldsymbol{0}$ yields
\begin{equation}\label{eq: det bound of z_*}
    \left|\det D_\xi\Phi_8\!\left(z_*,\underline{\tilde{\omega}}^4\right)\right|=12K^4\geq K^4.
\end{equation}With this preparation, we now present the main result of this section.

\begin{proposition}\label{lem: quantitative small set at z_*}
    There exist positive constants $C_1,C_2,C_3,C_4$, all independent of $K$ and $z_*$, such that for $n\geq 4$,
    \begin{equation}\label{eq: quantitative small set for z_*}
        \inf_{z\in B\left(z_*,C_1K^{-130}\right)}P^{(2),n}\left(z,\cdot\right)\geq c_1(K)^{\lfloor\frac{n}{4}\rfloor-1}C_2K^{-780}\mu(\cdot),
    \end{equation}where 
    \begin{equation*}
        c_1(K)=C_3K^{-764},\qquad \mu=K^{-288}\operatorname{Leb}\!\left(B\!\left(z_*,C_4K^{-55}\right)\cap \cdot\right).
    \end{equation*}
\end{proposition}

We begin by establishing the following lemma, which plays a central role in the proof of Proposition \ref{lem: quantitative small set at z_*}.
\begin{lemma}\label{lem: minorization condition}
    Assume that $z_0=(x_0,y_0)\in\left(\mathbb{T}^2\times\mathbb{T}^2\right)\setminus\Delta$, a phase sequence $\underline{\omega}_0^4=(\xi_0,\theta_0)\in\Omega_0^4$, and $s>0$ satisfy the following conditions
    \begin{enumerate}[label=(\arabic*), ref=(\arabic*)]
        \item\label{minorization: periodicity condition} $\Phi_8\!\left(z_0,\xi_0,\theta_0\right)=z_0$.
        \item\label{minorization: det bound} $\left|\det D_{\xi}\Phi_8\!\left(z_0,\xi_0,\theta_0\right)\right|\geq s$.
    \end{enumerate}Then, there exist constants $C,C'$, uniform with respect to $K,s,z_0$ and $\underline{\omega}_0^4$, such that the minorization inequality holds
    \begin{equation}\label{eq: uniform minorization}
         \inf_{z\in B\left(z_0,CK^{-135}s^2\right)} P^{(2),4}\left(z,\cdot\right)\geq Cs^{-1}\operatorname{Leb}(B\!\left(\theta_0,\rho)\right)\operatorname{Leb}\!\left(B\left(z_0,CK^{-63}s^2/2\right)\cap \cdot\right),
    \end{equation}where $B(\theta_0,\rho):=B\!\left(\theta_0,C'K^{-71}s^2\right)$.
\end{lemma}
\begin{proof}[Proof of Lemma \ref{lem: minorization condition}]
    Combining estimates \eqref{eq:first derivative omega_i}-\eqref{eq:second derivative z_j,j'} with Assumption \ref{minorization: det bound}, we observe that for any $\theta\in B\left(\theta_0,\frac{s}{128C^4K^{31}}\right)$, the following lower bound holds
    \begin{equation*}
        \left|\det D_\xi\Phi_8\!\left(z_0,(\xi_0,\theta)\right)\right|\geq \frac{s}{2}.
    \end{equation*}For a fixed $\theta$ in this ball, we define 
    \begin{equation*}
        z_\theta=(x_\theta,y_\theta):=\Phi_8\!\left(z_0,(\xi_0,\theta)\right).
    \end{equation*}In the following, $C$ denotes a generic positive constant independent of $K,s,z_0$ and $(\xi_0,\theta_0)$, whose value may change from line to line.
    
    We define the map $G_\theta\colon\!\left(\left(\mathbb{T}^2\times\mathbb{T}^2\right)\setminus\Delta\right)\times\Omega_0^2\rightarrow\left(\mathbb{T}^2\times\mathbb{T}^2\right)\setminus\Delta$ by 
    \begin{equation*}
        G_\theta\!\left(z,\xi\right):=\Phi_8\!\left(z,(\xi,\theta)\right)-\Phi_8\!\left(z_0,(\xi_0,\theta)\right),
    \end{equation*}so that $G_\theta\!\left(z_0,\xi_0\right)=0$. By setting $D=CK^9$ and $r=s/2$, the prior estimates \eqref{eq:first derivative omega_i}-\eqref{eq:second derivative z_j,j'} ensure that the requirements of Lemma \ref{lem: Quantitative implicit function theorem} are satisfied at $\left(z_0,\xi_0\right)$. Consequently, there exists a ball $B_1:=B\!\left(z_0,CK^{-135}s^2\right)$ and a $C^1$ function $H_\theta\colon B_1\rightarrow\Omega_0^2$ such that for any $z\in B_1$,
    \begin{equation}\label{eq: implicit function in B_1}
        \Phi_8\!\left(z,(H_\theta(z),\theta)\right)=\Phi_8\!\left(z_0,(\xi_0,\theta)\right)=z_\theta,\quad \left|\det D_\xi G_\theta\!\left(z,H_\theta(z)\right)\right|\geq \frac{1}{4}s.
    \end{equation}

    For a fixed $z\in B_1$, we define the local map $F_\theta:\Omega_0^2\rightarrow\left(\mathbb{T}^2\times\mathbb{T}^2\right)\setminus\Delta$ as 
    \begin{equation*}
        F_\theta(\cdot):=\Phi_8\!\left(z,\left(H_\theta(z)+\cdot,\theta\right)\right).
    \end{equation*}Set $D'=CK^9$ and $r'=s/4$. Combining \eqref{eq:first derivative omega_i}, \eqref{eq:second derivative omega_i,i'} with \eqref{eq: implicit function in B_1}, we verify the hypotheses of Lemma~\ref{lem: Quantitative inverse function theorem} for $F_\theta$. Therefore, there exist constants $C,C'>0$, independent of $K,z,z_0$ and $(\xi_0,\theta)$, such that:
    \begin{enumerate}[label=(\arabic*), ref=(\arabic*)]
        \item $F_\theta$ is injective on $B_2:=B(0,CK^{-36}s)$.
        \item $B\!\left(F_\theta(0),C'K^{-63}s^2\right)=B\!\left(z_\theta,C'K^{-63}s^2\right)=:B_3^\theta\subset F(B_2)$.
        \item For every $\xi\in B_2$, $\frac{1}{8}s\leq \left|DF_\theta(\xi)\right|\leq \frac{3}{8}s$ on $B_2$.
    \end{enumerate}Conditioning on the fixed $\theta$, for any Borel set $A\in \mathcal{B}\left(\left(\mathbb{T}^2\times\mathbb{T}^2\right)\setminus\Delta\right)$, we have
    \begin{align*}
        P^{(2),4}(z,A|\theta):=\mathbb{P}\!\left(\Phi_8(z,\xi,\theta)\in A\bigm|\theta\right)&=\int_{\Omega_0^2}\mathbf{1}_{A}\!\left(\Phi_8(z,\xi,\theta\right)\mathrm{d}\xi=\int_{\Omega_0^2}\mathbf{1}_{A}\left(F_\theta(\xi)\right)\mathrm{d}\xi\\
        &\geq\int_{B_2}\mathbf{1}_A\left(F_\theta(\xi)\right)\mathrm{d}\xi=\int_{F(B_2)}\mathbf{1}_A(y)\left|\det D_yF_\theta^{-1}(y)\right|\mathrm{d}y\\
        &\geq \int_{B_3^\theta}\mathbf{1}_A(y)\left|\det D_{\xi}F_\theta\!\left(F_\theta^{-1}(y)\right)\right|^{-1}\mathrm{d}y\\
        &\geq \frac{8}{3s}\operatorname{Leb}\!\left(B_3^\theta\cap A\right).
    \end{align*}

    Since the support $B_3^\theta$ varies with $\theta$, we seek a uniform measure, independent of $\theta$, to bound the transition kernel. Using the Lipschitz estimate \eqref{eq:first derivative omega_i}, we deduce that 
     \begin{equation*}
        |z_\theta-z_0|=\left|\Phi_8\!\left(z_0,(\xi_0,\theta)\right)-\Phi_8\!\left(z_0,\left(\xi_0,\theta_0\right)\right)\right|\leq CK^8\left|\theta-\theta_0\right|.
    \end{equation*}We now choose the radius $\rho:=\frac{1}{2C}C'K^{-71}s^2$. For any $\theta\in B(\theta_0,\rho)$, we claim that
    \begin{enumerate}[label=(\arabic*), ref=(\arabic*)]
        \item Since $\rho\le \frac{s}{128\,C^4K^{31}}$, the determinant lower bound remains valid:
        \begin{equation*}
            \left|\det D_\xi\Phi_8\!\left(z_0,(\xi_0,\theta)\right)\right|\geq \frac{s}{2}.
        \end{equation*}
        \item The center shift is controlled by $|z_\theta-z_0|\leq CK^8\rho\leq \frac{1}{2}C'K^{-63}s^2.$ Consequently, we have the inclusion
        \begin{equation*}
             B\!\left(z_0,C'K^{-63}s^2/2\right)\subset B\!\left(z_\theta,C'K^{-63}s^2\right)=B_3^\theta.
        \end{equation*}
    \end{enumerate}Therefore, for any $A\in \mathcal{B}\left(\left(\mathbb{T}^2\times\mathbb{T}^2\right)\setminus\Delta\right)$ and any $\theta\in B(\theta_0,\rho)$, we obtain that
    \begin{equation}\label{eq: conditional minorization uniform in theta}
        P^{(2),4}(z,\cdot|\theta)\geq \frac{8}{3s}\operatorname{Leb}\!\left(B_3^\theta\cap \cdot\right)\geq \frac{8}{3s}\operatorname{Leb}\!\left(B\!\left(z_0,C'K^{-63}s^2/2\right)\cap \cdot\right),
    \end{equation}which is uniform in $\theta$. Finally, integrating \eqref{eq: conditional minorization uniform in theta} over $\theta$ yields the desired uniform lower bound
    \begin{equation*}
         \inf_{z\in B\left(z_0,CK^{-135}s^2\right)} P^{(2),4}\left(z,\cdot\right)\geq \frac{8}{3s}\operatorname{Leb}\!\left(B(\theta_0,\rho)\right)\operatorname{Leb}\!\left(B\!\left(z_0,C'K^{-63}s^2/2\right)\cap \cdot\right).
    \end{equation*}By adjusting the constants $C,C'$ if necessary, the proof is complete. 
\end{proof}

\begin{proof}[Proof of Proposition \ref{lem: quantitative small set at z_*}]
    Consider the reference point $z_*=\left((0,0),(0,\pi)\right)$. We observe that the phase sequence $\underline{\tilde{\omega}}^4=\boldsymbol{0}\in\Omega_0^4$ satisfies the periodicity condition
    \begin{equation}\label{eq: periodicity of z_*}
        \Phi_8\!\left(z_*,\underline{\tilde{\omega}}^4\right)=z_*.
    \end{equation}Combining \eqref{eq: det bound of z_*} with \eqref{eq: periodicity of z_*}, we verify that the hypotheses of Lemma \ref{lem: minorization condition} hold with the substitutions
    \begin{equation*}
        z_0\mapsto z_*,\qquad (\xi_0,\theta_0)\mapsto\underline{\tilde{\omega}}^4,\qquad s\mapsto K^4.
    \end{equation*}Consequently, the uniform minorization \eqref{eq: uniform minorization} guarantees the existence of a constant $C>0$, independent of $K,z_*$, such that
    \begin{equation*}
        \inf_{z\in B\left(z_*,CK^{-127}\right)}P^{(2),4}(z,\cdot)\geq CK^{-4}\rho^4\operatorname{Leb}\!\left(B\!\left(z_*,CK^{-55}\right)\cap\cdot\right)=:\mu,
    \end{equation*}where $\rho\lesssim K^{-71}s^2=K^{-63}$ is the radius defined in Lemma \ref{lem: minorization condition}. Hereafter, $C$ denotes a generic positive constant independent of K, which may vary from line to line. A direct computation then yields
    \begin{equation*}
        \mu\!\left(B\!\left(z_*,CK^{-127}\right)\right)=CK^{-4-63*4-127*4}=CK^{-764}=:c_1(K).
    \end{equation*}

    By the Markov property and a straightforward induction, it follows that for any $n\geq 1$,
    \begin{equation}\label{eq: small set for 4n}
        \inf_{z\in B\left(z_*,CK^{-127}\right)}P^{(2),4n}\left(z,\cdot\right)\geq c_1(K)^{n-1}\mu(\cdot).
    \end{equation}To extend this result to an arbitrary step $n$, write $n=4m+i$ with $i\in\{1,2,3\}$. We restrict the starting point to a tighter ball $z\in B\!\left(z_*,CK^{-130}\right)$ and control the trajectory for the first $i$-steps. For any $\underline{\omega}^{i}\in B\!\left(\boldsymbol{0},CK^{-130}\right)\subset\Omega_0^i$, using the Lipschitz estimates \eqref{eq:first derivative omega_i} and \eqref{eq:first derivative z_j} together with $\Phi_{2i}(z_*,\mathbf{0})=z_*$, we obtain
    \begin{equation*}
        \left|\Phi_{2i}\!\left(z,\underline{\omega}^i\right)-z_*\right|\leq CK^i\left(|z-z_*|+\left|\underline{\omega}^i-\boldsymbol{0}\right|\right)\leq CK^{-127}.
    \end{equation*}This ensures that 
    \begin{equation*}
        P^{(2),i}\!\left(z,B\!\left(z_*,CK^{-127}\right)\right)\geq \mathbb{P}\!\left(\underline{\omega}^i\in B\!\left(\boldsymbol{0},CK^{-130}\right)\right)\geq CK^{-780}.
    \end{equation*}Finally, combining this transition probability with \eqref{eq: small set for 4n}, we conclude that for any $n\geq 4$
    \begin{equation*}
        \inf_{z\in B\left(z_*,CK^{-130}\right)}P^{(2),n}\!\left(z,\cdot\right)\geq c_1(K)^{\lfloor\frac{n}{4}\rfloor-1}CK^{-780}\mu(\cdot).
    \end{equation*}
\end{proof}

\subsubsection{Topological irreducibility of the two-point process}\label{subsubsec: Topological irreducibility of the two-point process}
For any initial configuration $(x,y)\in \left(\mathbb{T}^2\times\mathbb{T}^2\right)\setminus\Delta$ and any realization $\underline{\omega}^{\,n}=(\omega_1,\ldots,\omega_n)$, we denote the $n$-step trajectories of the system as
\begin{equation*}
    x_n=f_{\omega_n}\circ\cdots\circ f_{\omega_1}(x),\qquad y_n=f_{\omega_n}\circ\cdots\circ f_{\omega_1}(y).
\end{equation*}With this notation in place, we will establish the topological irreducibility of the two-point process by demonstrating its approximate controllability.
\begin{proposition}\label{prop: controllability of two point}
    Given $(x,y),(x_*,y_*)\in\left(\mathbb{T}^2\times\mathbb{T}^2\right)\setminus\Delta$ and $\epsilon>0$, there exist $N\in\mathbb{N}$ and $\underline{\omega}^N\in\Omega_0^N$ such that 
    \begin{equation*}
        \operatorname{dist}_{\mathbb{T}^2}\!\left(f_{\underline{\omega}^N}(x),x_*\right)+\operatorname{dist}_{\mathbb{T}^2}\!\left(f_{\underline{\omega}^N}(y),y_*\right)<\epsilon.
    \end{equation*}
\end{proposition}
We begin the proof with the following lemma allowing to approximately control to points of the form $\left((\bar{x}_1,0),(\bar{x}_1,\bar{x}_2)\right)\in\left(\mathbb{T}^2\times\mathbb{T}^2\right)\setminus\Delta$.
\begin{lemma}\label{lem:two point controllability of bar_x}
    Let $\epsilon,s>0$ and $\bar{x}_1,\bar{x}_2\in\mathbb{T}^2$ be target coordinates subject to the constraint $K\cos\!\left(\frac{\bar{x}_2}{2}\right)=-2\sqrt{2}\pi$. Then there exists a uniform time bound $N'=N'(\epsilon,s)$ such that for any initial pair $(x,y)\in\left(\mathbb{T}^2\times\mathbb{T}^2\right)\setminus\Delta$ with $\operatorname{dist}_{\mathbb{T}^2}(x,y)\geq s$, one can find a transition time $N_0 = N_0(x,y) \leq N'$ and $\underline{\omega}^{N_0}$ satifying
    \begin{equation*}
        \operatorname{dist}_{\mathbb{T}^2}\!\left(f_{\underline{\omega}^N}(x),x_*\right)+\operatorname{dist}_{\mathbb{T}^2}\!\left(f_{\underline{\omega}^N}(y),y_*\right)<\epsilon.
    \end{equation*}
\end{lemma}
\begin{proof}[Proof of Lemma \ref{lem:two point controllability of bar_x}]
    For a fixed $(x,y)\in \left(\mathbb{T}^2\times\mathbb{T}^2\right)\setminus\Delta$, we define the differences
    \begin{equation*}
        a_n:=y_{n,1}-x_{n,1},\qquad b_n:=y_{n,2}-x_{n,2}.
    \end{equation*}It follows from the definitions of $x_n$ and $y_n$ that the evolution of $(a_n,b_n)$ is governed by  
    \begin{align*}
        a_n&=a_{n-1}+2Kc_n\sin\left(\frac{b_{n-1}}{2}\right),\\
        b_n&=b_{n-1}+a_n,
    \end{align*}where we denote $c_n:=\cos\left(\frac{x_{n-1,2}+y_{n-1,2}}{2}-\omega_n^1\right)$. Note that for any given state, $c_n$ can take any value in $[-1,1]$ through an appropriate choice of the random phase $\omega_n^1$. 

    \vspace{0.3em}
    Assuming $b_0$ satisfies $2K\sin\left(\frac{b_0}{2}\right)>2\pi$, the term $2Kc_1\sin\left(\frac{b_0}{2}\right)$ can cover the entire interval $[0,2\pi)$. Thus, we can choose $\omega_1^1$ such that 
    \begin{equation*}
        a_1=-b_0+\bar{x}_2,
    \end{equation*}which directly implies that $b_1=\bar{x}_2$. Next, since $K\geq 4\pi$ and $\bar{x}_2$ is fixed, we have $2K\sin\left(\frac{b_1}{2}\right)>2\pi$. Hence we can choose $\omega_2^1$ such that 
    \begin{equation*}
        a_2=0,\qquad b_2=a_2+b_1=\bar{x}_2.
    \end{equation*}Simultaneously, we set $\omega_2^2=x_{1,2}+x_{2,1}$ so that
    \begin{equation*}
        x_{2,2}=x_{1,2}+x_{2,1}-\omega_2^2=0,\qquad y_{2,2}=x_{2,2}+b_2=\bar{x}_2,
    \end{equation*}yielding the configuration $(x_2,y_2)=\left((x_{2,1},0),(x_{2,1},\bar{x}_2)\right)$.
    
    For each $n\geq 3$, we choose $\omega_i$ recursively by
    \begin{equation}\label{eq:two-point choose omega_n}
        \omega_{n,1}=\frac{x_{n-1,2}+y_{n-1,2}}{2}+\frac{\pi}{2},\qquad \omega_n^2=x_{n-1,2}+x_{n,1}.
    \end{equation}Under this choice, the following identities hold for all $n\geq 3$,
    \begin{equation}\label{eq:a_n b_n hold}
        a_n\equiv 0, \quad b_n\equiv \bar{x}_2,\quad x_{n,2}\equiv0,\quad y_{n,2}\equiv\bar{x}_2.
    \end{equation}Combining \eqref{eq:two-point choose omega_n} with \eqref{eq:a_n b_n hold}, we obtain the horizontal increment
    \begin{equation}\label{eq:two point difference}
        x_{n,1}-x_{n-1,1}=K\sin\left(x_{n-1,2}-\omega_n^1\right)=K\sin\left(-\frac{\bar{x}_2}{2}-\frac{\pi}{2}\right)=2\sqrt{2}\pi,
    \end{equation}where in the last equality we used the choice of $\bar{x}_2$ satisfying $K\cos(\bar{x}_2/2)=-2\sqrt2\pi$. Iterating \eqref{eq:two point difference} yields
    \begin{equation*}
        x_{n,1}=y_{n,1}=x_{2,1}+(n-2)\cdot2\sqrt{2}\pi,\quad x_{n,2}=0,\quad y_{n,2}=\bar{x}_2.
    \end{equation*}Since $\frac{2\sqrt{2}\pi}{2\pi}=\sqrt{2}\notin \mathbb{Q}$, the sequence $\{x_{n,1}=x_{2,1}+(n-2)\cdot2\sqrt{2}\pi\}$ is dense in $\mathbb{T}^1$. Therefore, for any fixed $\bar{x}_1$ and $\epsilon>0$, there exist $N_0$ such that 
    \begin{equation*}
        \operatorname{dist}_{\mathbb{T}^1}\!\left(x_{N_0,1},\bar{x}_1\right)=\operatorname{dist}_{\mathbb{T}^1}\!\left(x_{2,1}+(N_0-2)\cdot 2\sqrt{2}\pi,\bar{x}_1\right)<\epsilon/2.
    \end{equation*}Consequently, we obtain 
    \begin{equation*}
        \operatorname{dist}_{\mathbb{T}^2}\!\left(f_{\underline{\omega}^{N_0}}(x),(\bar{x}_1,0)\right)+ \operatorname{dist}_{\mathbb{T}^2}\!\left(f_{\underline{\omega}^{N_0}}(y),(\bar{x}_1,\bar{x}_2)\right)<\epsilon,
    \end{equation*}as desired. To be more precise, the continued fraction approximants of $\sqrt{2}$ imply that for every $\epsilon>0$, the set 
    \begin{equation*}
        \left\{n\sqrt{2}\pmod  1:0\leq n\leq \left\lfloor\frac{3}{\epsilon}\right\rfloor+1\right\}
    \end{equation*}forms an $\epsilon$-net in $\mathbb{R}/\mathbb{Z}$. As a direct consequence, for any given $\epsilon>0$, one may choose 
    \begin{equation*}
         N_\epsilon=\left\lfloor\frac{12\pi}{\epsilon}\right\rfloor+1,
    \end{equation*}such that regardless of the positions of $x_{2,1}$ and $\bar{x}_1$,, there exists $0\leq n\leq N_\epsilon$ for which
    \begin{equation*}
        \operatorname{dist}_{\mathbb{T}^1}\!\left(x_{n,1},\bar{x}_1\right)<\epsilon/2.
    \end{equation*}

    We now consider the case where $2K\sin\left(\frac{b_0}{2}\right)\leq 2\pi$, which implies that $b_0$ lies in a small neighborhood of $0$ or $2\pi$. Without loss of generality, we assume $b_0$ is close to $0$; if $b_0=0$, the argument below can simply be initiated from $b_1=a_1>0$. Suppose $b_0>s>0$, then the term 
    \begin{equation*}
        a_1=a_0+2Kc_1\sin\left(\frac{b_0}{2}\right)
    \end{equation*}can cover an interval of the form $\left[a_0-\delta,a_0+\delta\right]$ with length at least $4K\sin\left(\frac{s}{2}\right)$. In particular, under the assumption $K\geq 4\pi$ and for $s$ sufficiently small, we have 
    \begin{equation*}
        4K\sin\left(\frac{s}{2}\right)\geq 2s.
    \end{equation*}

    Utilizing the density of the set $\{\frac{2k+1}{l}\pi:\ l,k\in\mathbb{N}\}$ in $[0,2\pi)$, for any $a\in\mathbb{T}^1$ and $s>0$, there exist $k,l\in\mathbb{N}$ with $l\leq \lfloor\frac{\pi}{s}\rfloor+1$ such that 
    \begin{equation*}
        \frac{2k+1}{l}\pi\in B(a,s).
    \end{equation*}Applying this with $a=a_0$, we choose $\omega_1^1$ such that 
    \begin{equation*}
        a_1=a_0+2Kc_1\sin\left(\frac{b_0}{2}\right)=\frac{2k+1}{l}\pi,\qquad b_1=b_0+\frac{2k+1}{l}\pi.
    \end{equation*}For subsequent steps $n\geq 2$, we set $\omega_n$ such that $c_n\equiv0$, which yields 
    \begin{equation*}
        a_n\equiv a_1,\qquad b_n=b_0+n\cdot\frac{2k+1}{l}\pi.
    \end{equation*}After $l$ iteration, where $l\leq \lfloor\frac{\pi}{s}\rfloor+1$, we obtain $b_l=\pi+b_0$. It follows that 
    \begin{equation*}
        2K\sin\left(\frac{b_l}{2}\right)=2K\sin\left(\frac{\pi}{2}+\frac{b_0}{2}\right)>2\pi, 
    \end{equation*}where the final inequality holds since $b_0$ is close to $0$. Thus, from time $l+1$ onward, we are reduced to the previously treated case, and the proof is complete.
\end{proof}
\begin{proof}[Proof of Proposition \ref{prop: controllability of two point}]
     Fix $(x,y),(x_*,y_*)\in \left(\mathbb{T}^2\times\mathbb{T}^2\right)\setminus\Delta$ and $\epsilon>0$. Let $g_\omega := f_\omega^{-1}$ denote the inverse map, explicitly evaluated as
     \begin{equation*}
         g_\omega:=f_\omega^{-1}=\begin{pmatrix}
            x_1-K\sin\left(x_2-x_1+\omega^2-\omega^1\right)\\[0.2em]
            x_2-x_1+\omega^2
        \end{pmatrix},
     \end{equation*}With a minor modification to the argument in the proof of Lemma~\ref{lem:two point controllability of bar_x} (applied to the inverse dynamics), there exist $N'=N'(\operatorname{dist}_{\mathbb{T}^2}\!(x_*,y_*))$ and a control sequence $\underline{\hat{\omega}}^{N'}$ such that, setting 
     \begin{equation*}
         \hat{x}:=\left(f_{\underline{\hat{\omega}}^{N'}}\right)^{-1}(x_*),\qquad \hat{y}:=\left(f_{\underline{\hat{\omega}}^{N'}}\right)^{-1}(y_*),
     \end{equation*}one has that 
     \begin{equation*}
         \hat{x}_1=\hat{y}_1,\qquad \hat{x}_2=0,\qquad \hat{y}_2=\bar{x}_2,
     \end{equation*}where $\bar{x}_2$ is same constant as in Lemma \ref{lem:two point controllability of bar_x}.

    Let $L'$ denote a uniform ($\underline{\omega}^{N'}$-independent) Lipschitz constant for the map $f_{\underline{\omega}^{N'}}$ and set $\epsilon_0:=\epsilon/2L'$. Applying Lemma \ref{lem:two point controllability of bar_x} with the substitutions
    \begin{equation*}
        \bar{x}_1\mapsto\hat{x}_1,\qquad \epsilon\mapsto\epsilon_0,
    \end{equation*}we obtain an integer $N_0'\in\mathbb{N}$ and $\underline{\omega}^{N_0'}$ such that 
    \begin{equation}\label{eq: reach hat_x}
        \operatorname{dist}_{\mathbb{T}^2}\!\left(f_{\underline{\omega}^{N_0'}}(x),(\hat{x}_1,0)\right)+\operatorname{dist}_{\mathbb{T}^2}\!\left(f_{\underline{\omega}^{N_0'}}(y),(\hat{x}_1,\bar{x}_2)\right)<\epsilon_0.
    \end{equation}Finally, we define the total time $N:=N_0'+N'$ and the concatenated control sequence $\underline{\omega}^N:=(\underline{\omega}^{N_0'},\underline{\hat{\omega}}^{N'})$. Since $f_{\underline{\tilde{\omega}}^{N'}}(\hat{x})=x_*$ and $f_{\underline{\tilde{\omega}}^{N'}}(\hat{y})=y_*$, the Lipschitz bound yields
    \begin{equation*}
        \operatorname{dist}_{\mathbb{T}^2}\!\left(f_{\underline{\omega}^N}(x),x_*\right)+\operatorname{dist}_{\mathbb{T}^2}\!\left(f_{\underline{\omega}^N}(y),y_*\right)\leq L'\epsilon_0<\epsilon,
    \end{equation*}which concludes the proof.
\end{proof}

\subsubsection{Quantitative reachability of $(x_*,y_*)$}\label{subsubsec: Quantitative reachability of z_*}
The primary objective of this section is to establish a quantitative lower bound for the transition probability into the small set $B\!\left(z_*,C_1K^{-130}\right)$, specified in Proposition \ref{lem: quantitative small set at z_*}. To formulate this result precisely, we introduce the near-diagonal set
\begin{equation*}
    \Delta(s):=\left\{(x,y):0<\operatorname{dist}_{\mathbb{T}^2}(x,y)<s\right\}.
\end{equation*}
\begin{lemma}\label{lem: quantitative reachability}
    Let $s>0$. For any $z\in\Delta(K^{-9}s)^c$, there exist $N_1=N_1(z)\in\mathbb{N}$ and $\underline{\omega}_0^{N_1}\in\Omega_0^{N_1}$ such that 
    \begin{equation*}
        \Phi_{2N_1}\!\left(z,\underline{\omega}_0^{N_1}\right)\in B\!\left(z_*,C_1K^{-130}\right),
    \end{equation*}where $C_1$ is the constant defined in Proposition \ref{lem: quantitative small set at z_*}. Furthermore, the following bounds hold:
    \begin{enumerate}[label=(\arabic*), ref=(\arabic*)]
        \item The hitting time admits the uniform bound
        \begin{equation}\label{eq: uniform upper bound for reaching z_*}
            \sup_{z\in\Delta(K^{-9}s)^c}N_1(z)\leq\left\lfloor\frac{\pi}{K^{-9}s}\right\rfloor+\left\lfloor\frac{12\pi}{C_1K^{-132}}\right\rfloor+4.
        \end{equation}
        \item There exists a constant $C>0$, independent of $K,z$, such that 
        \begin{equation}\label{eq: quantitative reachability of z_*}
            \inf_{z\in \Delta(K^{-9}s)^c}P^{(2),N_1}\!\left(z,B\!\left(z_*,C_1K^{-130}\right)\right)\geq p_K,
        \end{equation}where $p_K=K^{-CK^{264}}$.
    \end{enumerate}   
\end{lemma}

\begin{proof}
    For any fixed $z=(x,y)\in \Delta(K^{-9}s)^c$, the existence of a control time $N_1(z)$ and a phase sequence $\underline{\omega}_0^{N_1}$ such that driving the system to $z_*$ follows directly from Proposition \ref{prop: controllability of two point}.

    To derive a uniform upper bound for $N_1(z)$, we employ a backward-forward argument. First, consider the inverse map $g_\omega:=f_\omega^{-1}$ introduced in \eqref{eq: inverse map}. Following the strategy in Lemma \ref{lem:two point controllability of bar_x}, there exists a two-step phase sequence $\underline{\hat{\omega}}^2\in\Omega_0^2$ such that the pre-image point
    \begin{equation*}
         \hat{x}:=\left(f_{\underline{\hat{\omega}}^2}\right)^{-1}\!(x_*),\qquad \hat{y}:=\left(f_{\underline{\hat{\omega}}^2}\right)^{-1}\!(y_*),\qquad \hat{z}:=(\hat{x},\hat{y}),
    \end{equation*}satisfies the alignment condition
    \begin{equation*}
        \hat{x}_1=\hat{y}_1,\qquad \hat{x}_2=0,\qquad \hat{y}_2=\bar{x}_2,
    \end{equation*}where $\bar{x}_2$ is the constant defined in Lemma~\ref{lem:two point controllability of bar_x}. By the Lipschitz estimates \eqref{eq:first derivative omega_i} and \eqref{eq:first derivative z_j}, to steer the system into the small set $B(z*,C_1K^{-130})$, it suffices to reach the intermediate target ball $B\!\left(\hat{z},C_1K^{-132}/2C\right)$ within a bounded number of steps, where $C=C(2)$ is the constant from Lemma \ref{lem: priori estimates}.

    Second, since the initial state $z=(x,y)$ lies in the good set $\Delta(K^{-9}s)^c$, we may assume that the initial vertical separation
    \begin{equation*}
        |y_2-x_2|\geq \frac{K^{-9}s}{2}.
    \end{equation*}Following the argument in the proof of Lemma \ref{lem:two point controllability of bar_x}, it requires at most $\left\lfloor\frac{\pi}{K^{-9}s}\right\rfloor+1$ steps to steer $z$ to a state $\tilde{z}=(\tilde{x},\tilde{y})$ satisfying 
    \begin{equation*}
        2K\sin\!\left(\frac{\tilde{y}_2-\tilde{x}_2}{2}\right)>2\pi.
    \end{equation*}Subsequently, the same argument in Lemma \ref{lem:two point controllability of bar_x} shows that, in at most $\left\lfloor\frac{12\pi}{C_1K^{-132}}\right\rfloor+1$ steps, we can reach the intermediate ball $B\!\left(\hat{z},C_1K^{-132}/2C\right)$. Finally, including the $2$ steps for the final alignment (via $\underline{\hat{\omega}}^2$), we conclude that for any $z\in \Delta(K^{-9}s)^c$,
    \begin{equation*}
        N_1(z)\leq \left\lfloor\frac{\pi}{K^{-9}s}\right\rfloor+\left\lfloor\frac{12\pi}{C_1K^{-132}}\right\rfloor+4.
    \end{equation*}

    Moreover, the construction above actually yields
    \begin{equation*}
        \Phi_{2N_1}\!\left(z,\underline{\omega}_0^{N_1}\right)\in B\!\left(z_*,C_1K^{-130}/2\right),
    \end{equation*}since the intermediate set is chosen as $B\!\left(\hat{z},C_1K^{-132}/2C\right)$. Hence, by the Lipschitz continuity of the flow, for any $\underline{\omega}^{N_1}\in B\!\left(\underline{\omega}_0^{N_1},\frac{1}{2C}K^{-N_1}C_1K^{-130}\right)$, we have the estimate
    \begin{equation*}
        \left|\Phi_{2N_1}\!\left(z,\underline{\omega}^{N_1}\right)-\Phi_{2N_1}\!\left(z,\underline{\omega}_0^{N_1}\right)\right|\leq CK^{N_1}\left|\underline{\omega}^{N_1}-\underline{\omega}_0^{N_1}\right|\leq C_1K^{-130}/2.
    \end{equation*}Consequently, we obtain
    \begin{equation*}
        \left|\Phi_{2N_1}\!\left(z,\underline{\omega}^{N_1}\right)-z_*\right|\leq C_1K^{-130},
    \end{equation*}and therefore the transition probability is bounded from below by
    \begin{align*}
        P^{(2),N_1}\!\left(z,B\!\left(z_*,C_1K^{-130}\right)\right)&\geq \mathbb{P}\!\left(\underline{\omega}^{N_1}\in B\!\left(\underline{\omega}_0^{N_1},\frac{1}{2C}K^{-N_1}C_1K^{-130}\right)\right)\\[0.2em]
        &=\left(C'K^{-N_1}K^{-130}\right)^{2N_1}\geq p_K.
    \end{align*}
\end{proof}
The following result is a direct consequence of Lemmas \ref{lem: quantitative small set at z_*} and \ref{lem: quantitative reachability}.

\begin{lemma}\label{lem: quantitative small set for P^M}
    Let $s>0$. Define the uniform time $M$ by 
    \begin{equation}\label{eq: uniform time M}
        M:=\inf\!\left\{2k:2k\geq \left\lfloor\frac{\pi}{K^{-9}s}\right\rfloor+\left\lfloor\frac{12\pi}{C_1K^{-132}}\right\rfloor+8\right\}.
    \end{equation}Then, the following bound holds
    \begin{equation*}
        \inf_{z\in \Delta(K^{-9}s)^c}P^{(2),M}(z,\cdot)\geq p_Kc_1(K)^{\left\lfloor\frac{M}{4}\right\rfloor-1}C_2K^{-780}\mu(\cdot)
    \end{equation*}where $c_1(K),C_2,\mu$ are as defined in Proposition \ref{lem: quantitative small set at z_*}, and $p_K$ is the reachability probability from Lemma \ref{lem: quantitative reachability}.
\end{lemma}

\begin{proof}
    Fix $z=(x,y)\in \Delta(K^{-9}s)^c$. By Lemma \ref{lem: quantitative reachability}, there exists a hitting time $N_1(z)$ such that the system enters the small set $B\!\left(z_*,C_1K^{-130}\right)$ with probability at least $p_K$. Crucially, the uniform upper bound \eqref{eq: uniform upper bound for reaching z_*} ensures that
    \begin{equation*}
       N_1\leq M-4.
    \end{equation*}We define the remaining time $n:=M-N_1(z)$, satisfying $4\leq n\leq M$. Applying the Chapman-Kolmogorov property, and combining the lower bounds from \eqref{eq: quantitative reachability of z_*} and \eqref{eq: quantitative small set for z_*}, we obtain that 
    \begin{align*}
        P^{(2),M}(z,\cdot)&\geq p_Kc_1(K)^{\left\lfloor\frac{n}{4}\right\rfloor-1}C_2K^{-780}\mu(\cdot)\\
        &\geq p_Kc_1(K)^{\left\lfloor\frac{M}{4}\right\rfloor-1}C_2K^{-780}\mu(\cdot),
    \end{align*}which completes the proof.
\end{proof}

\subsection{Proof of Theorem \ref{thm:quantitative}}
Having verified all the hypotheses of Theorem \ref{thm: general framework for quantitative exponential mixing}, the proof of Theorem \ref{thm:quantitative} follows directly from the arguments presented in Section \ref{sec: Quantitative exponential mixing for incompressible RDS}. For the sake of completeness, we provide a brief sketch of the proof below. In what follows, $C>0$ denotes a generic constant independent of $K$, whose exact value may change from line to line.
\begin{proof}[Proof of Theorem \ref{thm:quantitative}]
    Adopting the notation for $p,\gamma,C_0$ and $V$ from Proposition \ref{prop: quantitative drift condition}, we define the constant $R_2:=\frac{4C_0K^{9p}}{1-\gamma}$. Since $V(x,y)=\operatorname{dist}_{\mathbb{T}^2}(x,y)^{-p}$, there exists a $K$-independent constant $s'>0$ such that
    \begin{equation*}
        \{V\leq R_2\}\subset \Delta(K^{-9}s')^c.
    \end{equation*}Applying Lemma \ref{lem: quantitative small set for P^M} to the region $\Delta(K^{-9}s')^c$, we can select an even integer $M=2k$ as in \eqref{eq: uniform time M}, bounded by $M\leq CK^{132}$, such that
    \begin{align}
        P^{(2),M}V&\leq \gamma^kV+C_0K^{9p}\frac{1-\gamma^k}{1-\gamma},\label{eq: drift condition for P^M}\\
        \inf_{z\in\Delta(K^{-9}s')^c}P^{(2),M}(z,\cdot)&\geq p_Kc_1(K)^{\left\lfloor \frac{M}{4}\right\rfloor-1}C_2K^{-780}\mu(\cdot).\notag
   \end{align}By recalling the explicit forms of $p_K,c_1(K)$ and $\mu$ (from Proposition \ref{lem: quantitative small set at z_*}), we can absorb all polynomial factors into the exponent and obtain
    \begin{equation}\label{eq: quantitative small set for P^M}
        \inf_{z\in\{V\leq R_2\}}P^{(2),M}(z,\cdot)\geq K^{-C'K^{264}}\nu_K(\cdot),
    \end{equation}where $\nu_K$ is the probability measure obtained by normalizing $\mu$.

    Proceeding as in Section \ref{subsec: Contraction Estimate for the Transition Kernel}, we introduce the following auxiliary constants
    \begin{align*}
        \alpha&:=K^{-C'K^{264}},\quad 
        L_{2n}:=C_0K^{9p}\frac{1-\gamma^n}{1-\gamma},\quad n\geq  1,\\
        \gamma_0&:=\frac{3}{4},\quad 
        \beta:=\frac{\alpha}{2L_M},\quad 
        \bar{\alpha}:=\max\!\left\{1-\frac{\alpha}{2},\ \frac{2+\beta R_2\gamma_0}{2+\beta R_2}\right\}.
    \end{align*}In view of Remark \ref{rmk: gamma tends to 0}, provided $K$ is sufficiently large, we have 
    \begin{equation*}
        \frac{3}{4}\in \left(\frac{1+\gamma^k}{2},1\right)=\left(\gamma^k+\frac{2L_M}{R_2},1\right),
    \end{equation*}Consequently, to streamline the presentation, we set $\gamma_0=\frac{3}{4}$ here rather than $\frac{3+\gamma}{4}$ used previously in Section \ref{subsec: Contraction Estimate for the Transition Kernel}. Applying the quantitative Harris theorem to $P^{(2),M}$ then implies that for any mean-zero function $\varphi\in H^{1}\!\left(\left(\mathbb{T}^2\times\mathbb{T}^2\right)\setminus\Delta\right)$,
    \begin{equation}\label{eq: contraction on P^{Mn}}
        \left\|P^{(2),Mn}\varphi\right\|_\beta \leq \bar{\alpha}^n\|\varphi\|_\beta.
    \end{equation}Furthermore, similar arguments to those in Section \ref{subsec: Contraction Estimate for the Transition Kernel} yield the bounds
    \begin{equation*}
        \left\|P^{(2)}\varphi\right\|_{\beta}\leq C^pK^p\|\varphi\|_\beta,\qquad \|P^{(2),2}\varphi\|_\beta\leq (1+\alpha)\|\varphi\|_\beta.
    \end{equation*}

    Fixing an arbitrary $n\in\mathbb{N}$, we write $n=lM+r$ for some integer $0\leq r<M$. Combining the contraction estimate \eqref{eq: contraction on P^{Mn}} for $P^{(2),lM}$ with the bound on the remaining $r$ steps, we have 
    \begin{equation}\label{eq: estimate for |P^n|_beta}
        \left\|P^{(2),n}\varphi\right\|_\beta\leq C^pK^{p}(1+\alpha)^k\bar{\alpha}^{\frac{n}{M}-1}\|\varphi\|_\beta.
    \end{equation}For sufficiently large $K$, $\alpha$ is small enough such that $(1+\alpha)^M\leq 2$. Moreover, a direct calculation shows
    \begin{equation*}
        \frac{2+\beta R_2\gamma_0}{2+\beta R_2}=\frac{3}{4}+\frac{1}{4}\cdot\frac{1}{1+(1-\gamma^k)^{-1}\alpha}\in \left(1-\frac{\alpha}{2},1-\frac{\alpha}{4}\right),
    \end{equation*}which implies $\bar{\alpha} \leq 1-\frac{\alpha}{4}$. Substituting these bounds into \eqref{eq: estimate for |P^n|_beta}, we obtain
    \begin{equation}\label{eq: drift condition for P^n_beta}
        \left\|P^{(2),n}\varphi\right\|_{\beta}\leq CK^{p}\left(1-\frac{\alpha}{4}\right)^{\frac{n}{M}}\|\varphi\|_\beta\leq CK^pe^{-\frac{1}{4}\alpha^2n}\|\varphi\|_\beta.
    \end{equation}  

    Next, let $\mathbb{Z}_0^2:=\mathbb{Z}^2\setminus\{(0,0)\}$ and denote by $\{e_m(x)=e^{im\cdot x}:m\in\mathbb{Z}^2\}$ the standard orthogonal basis for $L^2(\mathbb{T}^2)$. Consider mean-zero functions $\varphi,\psi\in H^1$ admitting the Fourier expansions
    \begin{equation*}
        \varphi=\sum_{m\in\mathbb{Z}_0^2}\hat{\varphi}_me_m,\qquad \psi=\sum_{m\in\mathbb{Z}_0^2}\hat{\psi}_me_m.
    \end{equation*}Adapting the criterion from Appendix \ref{subsec: Proof of Theorem 1.1} with $d=2$, we fix a constant $\zeta>0$ satisfying
    \begin{equation}\label{eq: requirements for zeta}
        2\zeta-\frac{1}{4}\alpha^2<0,\quad  6+\frac{2\zeta-\frac{1}{4}\alpha^2}{\zeta}<0,\quad 2+\frac{2\zeta-\frac{1}{4}\alpha^2}{\zeta q}<0.
    \end{equation}For instance, one may choose
    \begin{equation*}
        \zeta=\frac{\alpha^2}{10}\min\!\left\{\frac{1}{4},\frac{1}{1+q}\right\}.
    \end{equation*}For any $m,m'\in\mathbb{Z}_0^2$, we introduce random variables analogous to those in Appendix~\ref{subsec: Proof of Theorem 1.1}
    \begin{align*}
        N_{m,m'}&:=\max\!\left\{n\in\mathbb{N}:\left|\int e_m(x)e_{m'}\circ f_{\underline{\omega}}^n(x)\pi(\mathrm{d}x)\right|>e^{-\zeta n}\right\},\\
        M_0&:=\max\!\left\{|m|\vee|m'|:e^{\zeta N_{m,m'}}>|m|\ |m'|\right\},\\
        \hat{D}&:=\max_{|m|,|m'|\leq M_0} e^{\zeta N_{m,m'}}.
    \end{align*}Repeating the estimates derived in Appendix~\ref{subsec: Proof of Theorem 1.1}, we deduce that
    \begin{equation*}
        \left|\int \varphi(x)\psi\circ f_{\underline{\omega}}^n(x)\pi(\mathrm{d}x)\right|\leq \hat{D}(\underline{\omega})e^{-\zeta n}\|\varphi\|_{\dot{H}^3}\|\psi\|_{\dot{H}^3}.
    \end{equation*}An application of Lemma \ref{lem: control H^s' via H^s} then relaxes this regularity constraint, yielding the lower-regularity bound
    \begin{equation*}
        \left|\int f(x)g\circ f_{\underline{\omega}}^n(x)\pi(\mathrm{d}x)\right|\leq\hat{D}(\underline{\omega})e^{-\frac{\zeta}{3}n}\|f\|_{\dot{H}^1}\|g\|_{\dot{H}^1}.
    \end{equation*}

    Finally, it remains to bound the moments of $\hat{D}$. A similar estimation strategy to that of 
    Appendix \ref{subsec: Proof of Theorem 1.1} establishes
    \begin{equation*}
        \mathbb{E}[\hat{D}^q]\leq CK^p\!\left(1+\frac{\zeta q}{\frac{1}{4}\alpha^2-2\zeta(1+q)}\right)^{\frac{1}{2}}\sum_{l=1}^\infty l^{5+\frac{2\zeta-\frac{1}{4}\alpha^2}{\zeta}}.
    \end{equation*}By the second constraint imposed in \eqref{eq: requirements for zeta}, the exponent of the summation index $l$ is strictly less than $-1$, ensuring absolute convergence of the series. Therefore, we achieve the uniform bound $\mathbb{E}[\hat{D}^q]\leq CK^p$, which completes the proof.
\end{proof}

\section{Mixing and enhanced dissipation: proof of Theorem \ref{thm: exponential mixing}}\label{sec: Mixing and enhanced dissipation: proof of Theorem 1.6}
The goal of this section is to establish a qualitative version of exponential mixing and enhanced dissipation. To facilitate the proof, Section \ref{subsec: General framework for the proof of Theorem 1} recalls a general framework that reduces our main theorem to a set of verifiable dynamical conditions. The subsequent subsections are then devoted to verifying these hypotheses for our model.

\vspace{0.5em}

We begin by introducing necessary preliminaries on Lyapunov exponents. Let $f_\omega$ be a continuous RDS on a complete metric space $X$. The asymptotic growth of the derivative cocycle $D_xf_{\underline{\omega}}^n$ is characterized by the following theorem. For a detailed exposition, see \cite{Arnold-95}.

\begin{theorem}
    Let $f_{\underline{\omega}}^n$ be a continuous RDS with \textit{ergodic stationary} measure $\pi\in\mathcal{P}(X)$. Assume that the associated linear cocycle satisfies the integrability condition
    \begin{equation*}
        \int_{X\times\Omega_0} \left(\log^+\left\|D_xf_\omega\right\|+\log^+\left\|\left(D_xf_\omega\right)^{-1}\right\|\right)\mathrm{d}\mathbb{P}_0(\omega)\mathrm{d}\pi(x)<\infty.
    \end{equation*}Then there exist $r\in \{1,\ldots,d\}$, constant numbers $\lambda_1>\cdots>\lambda_r>-\infty$, and a filtration of subspaces 
    \begin{equation*}
        \mathbb{R}^d=F_1(\underline{\omega},x)\supsetneq \cdots\supsetneq F_r(\underline{\omega},x)\supsetneq F_{r+1}(\underline{\omega},x):=\{0\}
    \end{equation*}with the property that 
    \begin{equation*}
        \lambda_i=\lim_{n\rightarrow\infty}\frac{1}{n}\log\left\|D_xf_{\underline{\omega}}^n v\right\|,
    \end{equation*}for all $v\in F_i(\underline{\omega},x)\setminus F_{i+1}(\underline{\omega},x)$. Moreover, the dimensions $\dim F_i(\underline{\omega},x)$ are constant for $\mathbb{P}\times\pi$-almost every $(\underline{\omega},x)$.
\end{theorem}
The constant values $\lambda_i$ are called \textit{Lyapunov exponents} with multiplicities
\begin{equation*}
    m_i:=\dim F_i(\underline{\omega},x)-\dim F_{i+1}(\underline{\omega},x).
\end{equation*}The corresponding random sequence of subspaces $F_i$ is referred to as the \textit{Oseledets filtration}. By analogy with the identity
\begin{equation*}
    |\det B|=\prod_{i=1}^d \sigma_i(B), \quad B\in GL_d(\mathbb{R}),
\end{equation*}where $\sigma_i(B)$ are the singular values of $B$, we obtain the following limiting formula:
\begin{equation*}
    \lambda_{\Sigma}:=m_1\lambda_1+\cdots+m_r\lambda_r=\lim_{n\rightarrow\infty}\frac{1}{n}\log\left|\det(D_xf_{\underline{\omega}}^n)\right|, 
\end{equation*}where $\lambda_{\Sigma}$ is called the \textit{sum Lyapunov exponent}. In particular, if the random dynamical system is generated by incompressible (i.e., volume-preserving) diffeomorphisms, we have
\begin{equation*}
    \left|\det(D_xf_{\underline{\omega}}^n)\right|\equiv1,
\end{equation*}which implies $\lambda_{\Sigma}=0$. Furthermore, it is straightforward to check that
\begin{equation*}
    d\lambda_1\geq \lambda_{\Sigma}=0,
\end{equation*}where equality holds if and only if $r=1$ (that is, when all Lyapunov exponents coincide). A strictly positive top Lyapunov exponent $\lambda_1>0$ indicates exponential growth of typical tangent vectors and is deeply associated with mixing. Therefore, verifying the strict positivity $\lambda_1>0$ is of importance for our analysis.

\subsection{General framework for the proof of Theorem \ref{thm: exponential mixing}}\label{subsec: General framework for the proof of Theorem 1}In this section, we take the state space $X$ to be a complete metric space, and assume that the continuous RDS $f_\omega$ satisfies the Hypotheses (\ref{hypothese H_1})-(\ref{hypothese H_3}) formulated in Section \ref{sec: Quantitative exponential mixing for incompressible RDS}. The next proposition is a key ingredient in the proof of Theorem \ref{thm: exponential mixing}. Specifically, it establishes exponential mixing for the advection-diffusion equation~\eqref{eq:advection-diffusion} along the discrete times $t=2n$. This mixing result was first proved for the inviscid case $\nu=0$ in \cite{BBPS-22A,BCZG-23}, and was subsequently extended to the diffusive regime $0<\nu\ll 1$ in \cite{CIS-25}.
\begin{proposition}\label{prop:discrete time mixing}
    Let $f_{\omega,\nu}$ be a continuous RDS on $\mathbb{T}^2$ generated by SDE \eqref{eq:Lagrangian flow} with $\nu\geq 0$. Assume that the two-point process defined by $f_{\omega,0}$ is $V$-uniformly geometrically ergodic with $V\in L^1\left(\left(\mathbb{T}^2\times\mathbb{T}^2\right)\setminus\Delta\right)$. For any $q,s>0$ and any $\nu\geq0$ sufficiently small, there exist a random variable $\tilde{D}_\nu:\Omega\rightarrow [1,+\infty)$, and a $\nu$-independent constant $\lambda_s>0$ such that for all mean-zero $g,h\in H^s(\mathbb{T}^2)$, we have that 
    \begin{equation}
        \label{eq:discrete exponential mixing}
        \left|\int_{\mathbb{T}^2}g(x)h\circ f_{\underline{\omega},\nu}^n(x)\mathrm{d}x\right|\leq \tilde{D}_\nu(\underline{\omega})e^{-\lambda_sn}\|g\|_{\dot{H}^s}\|h\|_{\dot{H}^s}
    \end{equation}almost surely for all $n\in\mathbb{N}$. Furthermore, there exists a $\nu$-independent constant $C_q$ such that
    \begin{equation*}
        \mathbb{E}[\tilde{D}_\nu^q]\leq C_q.
    \end{equation*}
\end{proposition}
Proposition~\ref{prop:discrete time mixing} demonstrates that, uniformly in $\nu$, exponential mixing can be deduced directly from the uniform geometric ergodicity of the two-point process in the inviscid case $\nu=0$. Accordingly, for the remainder of this section, we restrict our attention to the inviscid dynamics and, for notational convenience, write $f_\omega := f_{\omega,0}$. 

To establish geometric ergodicity of the two-point process, the classical Harris theorem (see, e.g., \cite{book-MR2509253}) shows that it suffices to verify the following conditions:
\begin{itemize}
    \item [(1)] (\textbf{Small set}) There exist $n\in\mathbb{N}$, an open set $A\subset X$ and a positive measure $\nu_n$ on $X$ such that, for all $B\in\mathcal{B}(X)$ we have that 
    \begin{equation*}
        \inf_{x\in A}P^n(x,B)\geq \nu_n(B).
    \end{equation*}
    \item [(2)] (\textbf{Aperiodicity}) There exists $x_*\in X$ such that for every open neighborhood $A\ni x_*$ we have that $P(x_*,A)>0$.
    \item [(3)] (\textbf{Topological irreducibility}) For all $x\in X$ and nonempty open $A\subset X$, there exists $N=N(x,A)$ such that $P^N(x,A)>0$.
    \item [(4)] (\textbf{Lyapunov drift condition}) There exist a measurable function $V\colon X\rightarrow [0,+\infty)$, constants $\alpha\in (0,1)$, $C_0>0$, and a compact set $E\subset X$ satisfying the following drift condition
    \begin{equation*}
        PV(x)\leq \alpha V(x)+C_0\mathbf{1}_E(x),
    \end{equation*}for all $x\in X$.
\end{itemize}
In the subsequent lemmas, we formulate sufficient conditions that guarantee the small set, aperiodicity, and Lyapunov drift properties. The remaining ingredient, topological irreducibility for the randomized Chirikov standard map, will be established via a controllability argument. To this end, we introduce the following map. For each fixed $x\in X$, define $\Phi_x\colon \Omega_0^n\to X$ by
\begin{equation*}
    \Phi_x(\underline{\omega}^n):= f_{\omega_n}\circ\cdots\circ f_{\omega_1}(x), \qquad \underline{\omega}^n=(\omega_1,\ldots,\omega_n)\in\Omega_0^n.
\end{equation*}
\begin{lemma}[{\hspace{-0.05em}\cite[Proposition 3.1]{BCZG-23}, \textit{Small set}}]\label{lem: small set} Assume that the RDS $f_\omega$ satisfies Hypothesis~\textup{(\ref{hypothese H_1})}. Suppose there exist $n\geq 1$, $x_*\in X$ and $\underline{\omega}_*^n\in\Omega_0^n$ such that the following conditions hold:
    \begin{itemize}
        \item [(1)] There exist $\epsilon,\delta>0$ such that for all $\underline{\omega}^n\in\Omega_0^n$ satisfying $\left|\underline{\omega}^n-\underline{\omega}_*^n\right|\leq\epsilon$, one has
        \begin{equation*}
            \rho_0^n(\underline{\omega}^n)\geq\delta,
        \end{equation*}where $\rho_0^n$ denotes the density of the product probability measure $(\mathbb{P}_0)^n$.
        \item [(2)] $\Phi_{x_*}$ is a submersion at $\underline{\omega}^n=\underline{\omega}_*^n$. 
    \end{itemize}Then there exist an open set $A\subset X$ and a positive measure $\nu_n$ on $X$, absolutely continuous with respect to Lebesgue measure, such that for all $B\in\mathcal{B}(X)$ we have that 
    \begin{equation*}
        \inf_{x\in A} P^n(x,B)\geq \nu_n(B).
    \end{equation*}
\end{lemma}
\begin{lemma}[{\hspace{-0.05em}\cite[Lemma 3.2]{BCZG-23}, \textit{Aperiodicity}}]\label{lem: aperiodicity}
    Under Hypothesis \textup{(\ref{hypothese H_1})}, assume that there exist $\omega_*\in\Omega_0$ and $x_*\in X$ such that $f_{\omega_*}(x_*)=x_*$. Then for any open neighbourhood $A\ni x_*$, there holds that $P(x_*,A)>0$.
\end{lemma}

To verify the Lyapunov drift condition, certain auxiliary results are required. A crucial step is to rule out the degenerate case where $\lambda_1=0$. Henceforth, we restrict our discussion to the setting of the random Chirikov standard map, where $X=\mathbb{T}^2$. In this setting, we define the derivative cocycle $\Psi_x: \Omega_0^n \to SL_2(\mathbb{R})$ by
\begin{equation*}
    \Psi_x(\underline{\omega}^n) := \frac{1}{\left|\det D_x\Phi_x\right|^{1/2}} D_x\Phi_x = D_x\Phi_x,
\end{equation*}where the last equality holds since the flow $f_\omega$ is volume-preserving.
\begin{lemma}[{\hspace{-0.05em}\cite[Proposition 3.2]{BCZG-23}, \textit{Positive Lyapunov
 exponent}}]\label{lem:positivity}
    Let $f_\omega$ be a continuous RDS with an ergodic stationary measure $\pi$. Under Hypotheses \textup{(\ref{hypothese H_1})-(\ref{hypothese H_2})}, assume that there exist $n\in\mathbb{N}$, $\underline{\omega}_*^n\in\Omega_0^n$, and $x_*\in X$ satisfying the following properties:
    \begin{enumerate}[label=(\arabic*), ref=(\arabic*)]
    \item\label{property of positivity 1}There exist $\epsilon,\delta>0$ such that $\rho_0^n(\underline{\omega}^n)\geq \delta$ if $\left|\underline{\omega}^n-\underline{\omega}_*^n\right|\leq \epsilon$.
    \item\label{property of positivity 2} $\Phi_{x_*}$ is a submersion at $\underline{\omega}^n=\underline{\omega}_*^n$.
    \item\label{property of positivity 3} The restriction of $D_{\underline{\omega}_*^n}\Psi_{x_*}$ to $\ker\left(D_{\underline{\omega}_*^n}\Phi_{x_*}\right)$
        \begin{equation*}
            D_{\underline{\omega}_*^n}\Psi_{x_*} : \ker\bigl(D_{\underline{\omega}_*^n}\Phi_{x_*}\bigr) \rightarrow T_{\Psi_{x_*}(\underline{\omega}_*^n)} SL_2(\mathbb{R}),
        \end{equation*}is surjective.
    \end{enumerate}
    Then $\lambda_1>0$.
\end{lemma}

Building on Lemma \ref{lem:positivity}, the subsequent lemma provides a sufficient condition for establishing the Lyapunov drift inequality. This result originates from \cite{BCZG-23} and is summarized in \cite{coti-26}.
\begin{lemma}[{\hspace{-0.05em}\cite[Lemma 2.9]{coti-26}, \textit{Lyapunov drift condition}}]\label{lem:drift condition}
    Assume that both the one-point process and the projective process are geometrically ergodic, and that the top Lyapunov exponent $\lambda_1>0$. Then there exist constants $p>0$, $\gamma\in (0,1)$, and a function $V\in L^1\left(\left(\mathbb{T}^2\times\mathbb{T}^2\right)\setminus\Delta\right)$ with $V\geq 1$, such that
    \begin{equation}
        \operatorname{dist}_{\mathbb{T}^2}(x,y)^{-p}\lesssim V(x,y)\lesssim\operatorname{dist}_{\mathbb{T}^2}(x,y)^{-p}.
    \end{equation}Furthermore, the drift condition
    \begin{equation}
        P^{(2)}V(x,y)<\gamma V(x,y)
    \end{equation}holds for all $(x,y)$ in the complement of a compact set $E\subsetneq \left(\mathbb{T}^2\times\mathbb{T}^2\right)\setminus\Delta$.
\end{lemma}

\subsection{Small set and aperiodicity conditions}\label{subsec: small set and aperiodicity}
By applying Lemma \ref{lem: small set} and  \ref{lem: aperiodicity}, we now check the existence of small sets and aperiodicity for each of the one-point, projective and two-point Markov processes.
\subsubsection{The one-point process}
For $n=1$, the map $\Phi_x:\Omega_0\rightarrow\mathbb{T}^2$ takes the explicit form
\begin{equation*}
    \Phi_x(\omega)=\begin{pmatrix}
        x_1+K\sin(x_2-\omega^1)\\
        x_2+x_1+K\sin(x_2-\omega^1)-\omega^2
    \end{pmatrix}.
\end{equation*}Selecting $x_*=(0,0)$ and $\omega_*=(0,0)$ in Lemma \ref{lem: small set}, we obtain the Jacobian
\begin{equation*}
    D_{\omega_*}\Phi_x=\begin{pmatrix}
        -K & 0\\
        -K & -1
    \end{pmatrix},
\end{equation*}that has full rank. Therefore, the existence of small set for one-point process follows from Lemma~\ref{lem: small set}. Moreover, the chosen point $x_*$ is a fixed point for $f_{\omega_*}$, so that aperiodicity is given by Lemma \ref{lem: aperiodicity}.

\subsubsection{The projective process}

For any $(x,v)\in\mathbb{T}^2\times\mathbb{S}^1$, the mapping $\Phi_{x,v}:\Omega_0^n\rightarrow\mathbb{T}^2\times\mathbb{S}^1$ is defined by 
\begin{equation*}
    \Phi_{(x,v)}(\underline{\omega}^n)=\left(f_{\underline{\omega}^n}(x),\frac{D_xf_{\underline{\omega}^n}v}{\left|D_xf_{\underline{\omega}^n}v\right|}\right),
\end{equation*}where $f_{\underline{\omega}^n}:=f_{\omega_n}\circ\cdots\circ f_{\omega_1}$. In particular, the one-step Jacobian matrix $D_xf_\omega$ admits the following explicit form 
\begin{equation*}
    D_xf_{\omega}=\begin{pmatrix}
        1 & K\cos(x_2-\omega^1)\\
        1 & K\cos(x_2-\omega^1)+1
    \end{pmatrix}.
\end{equation*}To apply Lemma \eqref{lem: small set}, it suffices to consider the case $n=2$. We verify the submersion condition for $\Phi_{x_*,v_*}(\underline{\omega}^2)$ at the point $\underline{\omega}_*^2$, where 
\begin{equation*}
    x_*=(0,0),\quad v_*=(1,0),\quad \underline{\omega}_*^2=(\omega_{*,1},\omega_{*,2})=\left((0,0),(\pi/2,0)\right).
\end{equation*}A direct computation yields
\begin{equation*}
    D_{\underline{\omega}^2}\Phi_{(x_*,v_*)}\left(\underline{\omega}_*^2\right)=\begin{pmatrix}
        -K & 0 & 0 & 0\\
        -2K & -1 & 0 & -1\\
        \frac{-2K^2}{5\sqrt{5}} & \frac{-2K}{5\sqrt{5}} & \frac{-2K}{5\sqrt{5}} & 0\\[0.5em]
        \frac{K^2}{5\sqrt{5}} & \frac{K}{5\sqrt{5}} & \frac{K}{5\sqrt{5}} & 0
    \end{pmatrix}.
\end{equation*}It is straightforward to verify that this matrix has rank $3$, which equals $\dim(\mathbb{T}^2\times\mathbb{S}^1)$. Hence
$\Phi_{(x_*,v_*)}$ is a submersion at $\underline{\omega}_*^{\,2}$, and
Lemma~\ref{lem: small set} yields the desired small-set condition.

\vspace{0.3em}
To verify aperiodicity of the projective process, we can select 
\begin{equation*}
    x_*=(0,0),\quad v_*=\frac{1}{\sqrt{\frac{K^2+2K+2-K\sqrt{K^2+4K}}{2}}}\left(\frac{-K+\sqrt{K^2+4K}}{2},1\right), \quad \omega_*=(0,0).
\end{equation*}A direct computation shows that
\begin{equation*}
    \left(f_{\omega_*}(x_*),\frac{D_{x_*}f_{\omega_*}v}{\left|D_{x_*}f_{\omega_*}v\right|}\right)=(x_*,v_*),
\end{equation*}which satisfies the hypotheses of Lemma \ref{lem: aperiodicity}.

\subsubsection{The two-point process}
We apply Lemma \ref{lem: small set} with $n=3$, since the case $n=2$ is degenerate and does not satisfy the submersion condition. We choose
\begin{equation*}
    x_*=(0,0), \quad y_*=(\pi,\pi),\quad \underline{\omega}_*^3=\left((0,0),(0,0),(0,0)\right),
\end{equation*}so that the corresponding Jacobian
\begin{equation*}
    D_{\underline{\omega}_*^3}\Phi_{(x_*,y_*)}=\begin{pmatrix}
        -K(K^2+3K+1) & -K(K+2) & -K(K+1) & -K & -K & 0\\
        -K(K^2+4K+3) & -K^2-3K-1 & -K(K+2) & -K-1 & -K & -1\\
        -K(K^2+K-1) & K^2 & (K-1)K & K & K & 0\\
        -K(K^2-3) & K^2-K-1 & (K-2)K & K-1 & K & -1
    \end{pmatrix},
\end{equation*}has full rank 4 as desired. 

For Lemma \ref{lem: aperiodicity}, we check directly that 
\begin{equation*}
    x_*=(0,0),\quad y_*=(0,\pi),\quad \omega_*=(0,0)
\end{equation*}has the property that $\left(f_{\omega_*}(x_*),f_{\omega_*}(y_*)\right)=(x_*,y_*)$.
\subsection{Positivity of Lyapunov exponent}\label{subsec: Positive Lyapunov exponent}
In the setting of the randomized Chirikov map, property \ref{property of positivity 1} in Lemma \ref{lem:positivity} is automatically satisfied. Consequently, it suffices to verify that there exist $n\geq 1$ and $\underline{\omega}_*^n\in\Omega_0^n$ for which properties \ref{property of positivity 2} and \ref{property of positivity 3} hold.

We observe that if property \ref{property of positivity 2} holds at $(\underline{\omega}_*^n,x_*)$, the kernel of $D_{\underline{\omega}_*^n}\Phi_{x_*}$ has dimension at most $2n-2$. Given that the target space of $D_{\underline{\omega}_*^n}\Psi_{x_*}$ is three-dimensional, property \ref{property of positivity 3} requires the dimension of the kernel be at least 3 (i.e., $2n-2\geq 3$), which implies $n\geq 3$. However, a direct computation shows that properties \ref{property of positivity 2} and \ref{property of positivity 3} cannot be satisfied simultaneously when
$n=3$. Therefore, we will verify these properties for $n=4$.

To start, let 
\begin{equation*}
    \begin{aligned}
        x_*&=\left(\frac{\pi}{2},\pi\right),\\[0.3em]
        \underline{\omega}_*^4&=\left(\left(0,\frac{\pi}{2}-1\right),\left(\frac{\pi}{2}+1,\frac{\pi}{2}+K\right),\left(\frac{\pi}{2}+1,\frac{\pi}{2}+2K\right),\left(\frac{\pi}{2}+1,\frac{\pi}{2}+3K\right)\right),
    \end{aligned}
\end{equation*}where each phase is understood modulo $2\pi$, and hence lies in $[0,2\pi)$. We obtain the matrix 
\begin{equation*}
    D_{\underline{\omega}_*^4}\Phi_{x_*}=\begin{pmatrix}
        K & 0 & 0 & 0 & 0 & 0 & 0 & 0\\
        4K & -1 & 0 & -1 & 0 & -1 & 0 & -1
    \end{pmatrix}.
\end{equation*}
Since $\mathrm{rank}\left(D_{\underline{\omega}_*^4}\Phi_{x_*}\right)=2$, property \ref{property of positivity 2} holds. Therefore $\ker D_{\underline{\omega}_*^4}\Phi_{x_*}$ is 6-dimensional and a basis is given by the columns of the following matrix
\begin{equation*}
    \mathcal{K}=\begin{pmatrix}
        0 & 0 & 0 & 0 & 0 & 0\\
        0 & -1 & 0 & -1 & 0 & -1\\
        1 & 0 & 0 & 0 & 0 & 0\\
        0 & 1 & 0 & 0 & 0 & 0\\
        0 & 0 & 1 & 0 & 0 & 0\\
        0 & 0 & 0 & 1 & 0 & 0\\
        0 & 0 & 0 & 0 & 1 & 0\\
        0 & 0 & 0 & 0 & 0 & 1
    \end{pmatrix}.
\end{equation*}

\vspace{0.3em}
Moreover, we identify the space of $2\times 2$ real matrices with $\mathbb{R}^4$ via the standard parametrization
\begin{equation*}
    \begin{pmatrix}
        a & b\\
        c & d
    \end{pmatrix}\mapsto \begin{pmatrix}
        a\\
        b\\
        c\\
        d
    \end{pmatrix}.
\end{equation*}Under this identification, we obtain
\begin{equation*}
    \mathcal{M}:=D_{\underline{\omega}_*^{\,4}}\Psi_{x_*}=\scalebox{0.75}{$
    \left(
    \begin{array}{cccccccc}
        -14K^2 & 6K & K & 5K & 2K & 3K & 3K & 0\\
        2K^2(7K-3) & 3K(1-2K) & K(1-K) & K(2-5K) & K(1-2K) & K(1-3K) & K(1-3K) & 0\\
        -20K^2 & 10K & 3K & 7K & 4K & 3K & 3K & 0\\
        10K^2(2K-1) & 2K(3-5K) & 3K(1-K) & K(3-7K) & 2K(1-2K) & K(1-3K) & K(1-3K) & 0
    \end{array}
    \right)$}.  
\end{equation*}To prove that $\mathcal{M}$ maps $\ker D_{\underline{\omega}_*^4}\Phi_{x_*}$ surjectively onto $T_{\Psi_{x_*}(\underline{\omega}_*^4)}SL_2(\mathbb{R})$, it suffices to verify that 
\begin{equation*}
    \mathcal{M}\mathcal{K}=\scalebox{0.85}{$
    \left(
    \begin{array}{cccccccc}
        K & -K & 2K & -3K & 3K & -6K\\
        K(1-K) & K(K-1) & K(1-2K) & K(3K-2) & K(1-3K) & 3K(2K-1)\\
        3K & -3K & 4K & -7K & 3K & -10K\\
        3K(1-K) & 3K(K-1) & 2K(1-2K) & K(7K-5) & K(1-3K) & 2K(5K-3)
    \end{array}
    \right)$}  
\end{equation*}
has rank 3. Therefore, properties \ref{property of positivity 1}-\ref{property of positivity 3} hold at the chosen point $(x_*,\underline{\omega}_*^4)$, and Lemma \ref{lem:positivity} yields the positivity of the top Lyapunov exponent $\lambda_1$.
\begin{rmk}
    We have verified the preceding computations using the computer algebra system Mathematica.
\end{rmk}

\subsection{Topological irreducibility}\label{subsec: Topological irreducibility}
In this section, we prove topological irreducibility via a controllability
argument. More precisely, given any initial point and a nonempty open set, we show that the dynamics can reach the target open set by choosing an appropriate finite sequence of random phases $\{\omega_i\}_{i=1}^n$.

\subsubsection{The one-point process}
We begin by establishing exact controllability for the one-point dynamics generated by the RDS $f_{\underline{\omega}}^{\,n}$.
\begin{proposition}
    Given $x,y\in\mathbb{T}^2$, there exists $\omega\in\Omega_0$ such that 
    \begin{equation*}
        f_\omega(x)=y,
    \end{equation*}namely the one-point process is exactly controllable.
\end{proposition}
\begin{proof}
    Note that one can choose $\omega=(\omega^1,\omega^2)$ such that 
    \begin{equation*}
        \sin\left(x_{0,2}-\omega^1\right)=\frac{y_{0,1}-x_{0,1}}{K},\qquad \omega^2=x_{0,2}+y_{0,1}-y_{0,2} \pmod{2\pi}.
    \end{equation*}With this choice, we obtain
    \begin{align*}
        x_{1,1}&=x_{0,1}+K\sin\left(x_{0,2}-\omega^1\right)=y_{0,1},\\
        x_{1,2}&=x_{0,2}+x_{0,1}+K\sin\left(x_{0,2}-\omega^1\right)-\omega^2=y_{0,2}.
    \end{align*}Therefore $f_\omega(x)=y$, as claimed.
\end{proof}

\subsubsection{The projective process}
We define the projective dynamics $\hat{f}_\omega:\mathbb{T}^2\times\mathbb{S}^1\rightarrow\mathbb{T}^2\times\mathbb{S}^1$ by
\begin{equation*}
    \hat{f}_{\omega}(x,v):=\left(f_\omega(x),\frac{D_xf_{\omega}v}{\left|D_xf_\omega v\right|}\right),
\end{equation*}For $\underline{\omega}^n=(\omega_1,\ldots,\omega_n)\in\Omega_0^n$, we denote its $n$-step iterate by $\hat{f}_{\underline{\omega}^n}:=\hat{f}_{\omega_n}\circ\cdots\circ\hat{f}_{\omega_1}$. Recall that the associated projective Markov transition kernel is given by 
\begin{equation*}
    \hat{P}\left((x,v),\hat{A}\right)=\mathbb{P}\left(\hat{f}_\omega(x,v)\in\hat{A}\right),\qquad \hat{A}\in\mathcal{B}\left(\mathbb{T}^2\times\mathbb{S}^1\right).
\end{equation*}The following proposition shows that the RDS $\hat{f}_\omega$ is approximately
controllable, which in turn implies topological irreducibility of $\hat{P}$.
\begin{proposition}\label{prop:controllability of projective}
    For any $(x,v),(x_*,v_*)\in \mathbb{T}^2\times\mathbb{S}^1$ and $\epsilon>0$, there exist $N\in\mathbb{N}$ and $\underline{\omega}^N=(\omega_1,\ldots,\omega_N)$ such that
    \begin{equation*}
        \operatorname{dist}_{\mathbb{T}^2\times\mathbb{S}^1}\!\left(\hat{f}_{\underline{\omega}^N}(x,v),(x_*,v_*)\right)<\epsilon.
    \end{equation*}
\end{proposition}
The following auxiliary result is a key ingredient in the argument.
\begin{lemma}\label{lem:projective controllability of bar_x}
    Given any $(x,v)\in\mathbb{T}^2\times\mathbb{S}^1$, $\bar{x}\in\mathbb{T}^2$ and $\epsilon>0$, there exists $N_1\in\mathbb{N}$ and $\underline{\omega}^{N_1}\in\Omega_0^{N_1}$ such that 
    \begin{equation*}
        \operatorname{dist}_{\mathbb{T}^2\times\mathbb{S}^1}\!\left(\hat{f}_{\underline{\omega}^{N_1}}(x,v),\left(\bar{x},(0,1)\right)\right)<\epsilon.
    \end{equation*}
\end{lemma}
\begin{proof}[Proof of Lemma \ref{lem:projective controllability of bar_x}]
   The proof is divided into two cases: $\frac{K}{\pi}\in\mathbb{Q}$ and $\frac{K}{\pi}\notin\mathbb{Q}$.

   \vspace{0.3em}
   Assume that $\frac{K}{\pi}=\frac{q}{p}\in\mathbb{Q}$. We first choose $\omega_1=(\omega_1^1,\omega_1^2)$ such that
   \begin{equation*}
        \sin\left(x_{0,2}-\omega_1^1\right)=\frac{\bar{x}_1-x_{0,1}}{K},\qquad \omega_1^2=x_{0,2}+\bar{x}_1-\bar{x}_2 \pmod{2\pi}.
   \end{equation*}This choice ensures that $x_1=\bar{x}\in\mathbb{T}^2$. For the moment assume that $w_{1,1}>0$ (hence $v_{1,1}>0$ as well). For each $i\geq 2$, we choose $\omega_i=(\omega_{i,1},\omega_{i,2})$ recursively 
   \begin{equation}\label{eq:omega-recursively}
       \sin\left(x_{i-1,2}-\omega_i^1\right)=1,\qquad \omega_i^2=x_{i-1,2}+x_{i,1}-\bar{x}_2,
   \end{equation}so that the Jacobian satisfies
   \begin{equation*}
       D_{x_{i-1}}f_{\omega_i}=\begin{pmatrix}
           1 & 0\\
           1 & 1
       \end{pmatrix}.
   \end{equation*}Consequently, we obtain
   \begin{equation}
       w_{i,1}=w_{i-1,1},\qquad w_{i,2}=w_{i,1}+w_{i-1,2}.\label{eq:w-recursion}
   \end{equation}Iterating \eqref{eq:w-recursion} yields, for all $n\geq 2$,
   \begin{equation*}
       w_{n,1}=w_{1,1},\qquad w_{n,2}=(n-1)w_{1,1}+w_{1,2}.
   \end{equation*}
   
   Since $w_{1,1}>0$, given any $k\in\mathbb{R}^+$ there exists $N\in\mathbb{N}$ such that for all $n\geq N$
   \begin{equation*}
       \frac{w_{n,2}}{w_{n,1}}=\frac{v_{n,2}}{v_{n,1}}>k.
   \end{equation*}Let us choose $k$ large enough so that
   \begin{equation*}
       \operatorname{dist}_{\mathbb{S}^1}\!\left(\frac{1}{\sqrt{k^2+1}}(1,k),(0,1)\right)<\epsilon.
   \end{equation*}Then, for all $n\geq N$, one has that
   \begin{equation*}
       \operatorname{dist}_{\mathbb{S}^1}\!\left((v_{n,1},v_{n,2}),(0,1)\right)\leq   \operatorname{dist}_{\mathbb{S}^1}\!\left(\frac{1}{\sqrt{k^2+1}}(1,k),(0,1)\right)<\epsilon.
   \end{equation*}
   Finally, set $N_1=1+2pN$. Since $\frac{K}{\pi}=\frac{q}{p}$, the horizontal phase evolution is periodic with period $2p$, and by construction we have
   \begin{equation*}
       x_{N_1}=\bar{x},\qquad  \operatorname{dist}_{\mathbb{S}^1}\!\left(v_{N_1},(0,1)\right)<\epsilon.
   \end{equation*}Therefore,
   \begin{equation*}
        \operatorname{dist}_{\mathbb{T}^2\times\mathbb{S}^1}\!\left(\hat{f}_{\underline{\omega}^{N_1}}(x,v),\left(\bar{x},(0,1)\right)\right)<\epsilon,
   \end{equation*}as desired. It remains to treat the case $w_{1,1}\leq 0$. We continue to choose the phases $\omega_i$ according to \eqref{eq:omega-recursively} up to some time $n_1$, where $n_1$ is chosen so that 
   \begin{equation*}
       w_{n_1,2}=(n_1-1)w_{1,1}+w_{1,2}<\frac{1}{K}w_{1,1}<0.
   \end{equation*}At the next step we select $\omega_{n_1+1}$ so that
   \begin{equation*}
       \cos\left(x_{n_1,2}-\omega_{n_1+1}^1\right)=-1,\qquad \omega_{n_1+1}^2=x_{n_1,2}+x_{n_1+1,1}-\bar{x}_2 \pmod{2\pi}.
   \end{equation*}A direct computation then shows that 
   \begin{equation*}
       w_{n_1+1,1}=w_{1,1}-Kw_{n_1,2}>0.
   \end{equation*}From time $n_1+1$, we may repeat the argument in the case $w_{1,1}>0$ to conclude the proof.

    \vspace{0.3em}
    We now turn to the case $\frac{K}{\pi}\notin\mathbb{Q}$. Since the set $\{nK \pmod{2\pi}: n\in\mathbb{N}\}$ is dense in $[0,2\pi)$, we can approximate any prescribed horizontal phase by a suitable iterate. Arguing as in the rational case, we may assume that $w_{0,1}>0$ (and hence $v_{0,1}>0$). For all $i\geq 1$, we choose $\omega_i$ same as \eqref{eq:omega-recursively}. Then there exists $N\in\mathbb{N}$ such that for every $n\ge N$,
    \begin{equation}\label{eq:v_n control}
        \operatorname{dist}_{\mathbb{S}^1}\!\left(v_n,(0,1)\right)<\frac{\epsilon}{2}.
    \end{equation}Moreover, along the same sequence we have
    \begin{equation*}
        x_{n,1}=x_{0,1}+nK, \qquad x_{n,2}=\bar{x}_2 \pmod{2\pi}.
    \end{equation*}

    By density of $\{nK \pmod{2\pi}:n\in\mathbb{N}\}$, we can choose an integer $N_1\ge N$ such that
    \begin{equation}\label{eq:x_N_1 control}
        \operatorname{dist}_{\mathbb{T}^1}\!\left(x_{N_1,1},\bar{x}_1\right)=\operatorname{dist}_{\mathbb{T}^1}\!\left(x_{0,1}+N_1K,\bar{x}_1\right)<\epsilon/2.
    \end{equation}Combining \eqref{eq:v_n control} and \eqref{eq:x_N_1 control}, we obtain 
   \begin{equation*}
        \operatorname{dist}_{\mathbb{T}^2\times\mathbb{S}^1}\!\left(\hat{f}_{\underline{\omega}^{N_1}}(x,v),\left(\bar{x},(0,1)\right)\right)<\epsilon.
   \end{equation*}
\end{proof}

\begin{proof}[Proof of Proposition \ref{prop:controllability of projective}]
    Fix $(x,v),(x_*,v_*)\in\mathbb{T}^2\times\mathbb{S}^1$ and $\epsilon>0$. Let $g_\omega:=f_\omega^{-1}$, which admits the explicit form
    \begin{equation}\label{eq: inverse map}
        g_\omega(x)=\begin{pmatrix}
            x_1-K\sin\left(x_2-x_1+\omega^2-\omega^1\right)\\[0.2em]
            x_2-x_1+\omega^2
        \end{pmatrix},
    \end{equation}and hence
    \begin{equation*}
        D_xg_\omega=\begin{pmatrix}
            1+K\cos\left(x_2-x_1+\omega^2-\omega^1\right) & -K\cos\left(x_2-x_1+\omega^2-\omega^1\right)\\
            -1 & 1
        \end{pmatrix}.
    \end{equation*}We first show that there exist $N'\in\mathbb{N}$ and $\underline{\tilde{\omega}}^{N'}\in\Omega_0^{N'}$ such that 
    \begin{equation*}
        D_xg_{\underline{\tilde{\omega}}^{N'}}(v_*)=(0,1).
    \end{equation*}

    For $n\geq 0$, we define 
    \begin{equation*}
       \tilde{x}_{*,n}:=g_{\underline{\tilde{\omega}}^n}(x_*),\qquad \tilde{w}_{*,n}:=D_xg_{\underline{\tilde{\omega}}^n}v_*,\qquad \tilde{v}_{*,n}:=\frac{D_xg_{\underline{\tilde{\omega}}^n}v_*}{\left|D_xg_{\underline{\tilde{\omega}}^n}v_*\right|},
    \end{equation*}we also write $\left[\tilde{x}_{*,n}\right]_i$ for the $i$-th coordinate of $\tilde{x}_{*,n}$. We may assume that $[v_{*,0}]_1<0$, if not, the same proof works with minor modifications. For each $i\geq 1$, we choose $\tilde{\omega}_i$ so that 
    \begin{equation*}
       \cos\left(\left[\tilde{x}_{*,i-1}\right]_2-\left[\tilde{x}_{*,i-1}\right]_1+\tilde{\omega}_i^2-\tilde{\omega}_i^1\right)=0,
    \end{equation*}Then we obtain
    \begin{equation*}
        D_{\tilde{x}_{*,i-1}}g_{\tilde{\omega}_i}=\begin{pmatrix}
           1 & 0\\
           -1 & 1
       \end{pmatrix},
    \end{equation*}and hence, by iteration, 
    \begin{equation}\label{eq:representation of w_*,i}
        \left[\tilde{w}_{*,i}\right]_1=\left[\tilde{w}_{*,0}\right]_1,\qquad \left[\tilde{w}_{*,i}\right]_2=-i\left[\tilde{w}_{*,0}\right]_1+\left[\tilde{w}_{*,0}\right]_2.
    \end{equation}Moreover, the first component of $\tilde{w}_{*,n}=D_{\tilde{x}_{*,n-1}}g_{\tilde{\omega}_n}\left(\tilde{w}_{*,n-1}\right)$ satisfies 
    \begin{equation*}
        \left[\tilde{w}_{*,n}\right]_1=\left[\tilde{w}_{*,n-1}\right]_1+K\cos\left(\left[x_{*,n-1}\right]_2-\left[x_{*,n-1}\right]_1+\tilde{\omega}_n^2-\tilde{\omega}_n^1\right)\left(\left[\tilde{w}_{*,n-1}\right]_1-\left[\tilde{w}_{*,n-1}\right]_2\right)
    \end{equation*}Therefore, $\left[\tilde{w}_{*,n}\right]_1=0$ is established by choosing $\tilde{\omega}_n$ such that 
    \begin{equation}\label{eq:choose omega_n}
        \cos\left(\left[\tilde{x}_{*,n-1}\right]_2-\left[\tilde{x}_{*,n-1}\right]_1+\tilde{\omega}_n^2-\tilde{\omega}_n^1\right)=\frac{\left[\tilde{w}_{*,n-1}\right]_1}{K\left(\left[\tilde{w}_{*,n-1}\right]_2-\left[\tilde{w}_{*,n-1}\right]_1\right)}
    \end{equation}Combining \eqref{eq:representation of w_*,i} with \eqref{eq:choose omega_n}, we deduce that there exists $N'\in\mathbb{N}$ such that 
    \begin{equation*}
        \left[\tilde{w}_{*,N'}\right]_2=-N'\left[w_{*,0}\right]_1+\left[w_{*,0}\right]_2>0,
    \end{equation*}and the right-hand side of \eqref{eq:choose omega_n} (with $n=N'$) lies in $[-1,1]$. With this choice of $\tilde{\omega}_{N'}$ we obtain 
    \begin{equation*}
        \left[\tilde{w}_{*,N'}\right]_1=0,\qquad \left[\tilde{w}_{*,N'}\right]_2>0,
    \end{equation*}and therefore $\tilde{v}_{*,N'}=(0,1)$ as desired.

    Set $\bar{x}:=g_{\underline{\tilde{\omega}}^{N'}}(x_*)$ and let $L'>0$ be an $\underline{\omega}^{N'}$-uniform upper bound on $\|\hat{f}_{\underline{\omega}^{N'}}\|_{\text{Lip}}$. Define $\epsilon_0:=\epsilon/2L'$. By Lemma \ref{lem:projective controllability of bar_x}, there exist $N_1\in\mathbb{N}$ and $\underline{\omega}^{N_1}$ such that
    \begin{equation*}
        \operatorname{dist}_{\mathbb{T}^2\times\mathbb{S}^1}\!\left(\hat{f}_{\underline{\omega}^{N_1}}(x,v),(\bar{x},(0,1))\right)<\epsilon_0.
    \end{equation*}We then set $\omega_{N_1+i}=\tilde{\omega}_{N'-i}$ for $1\leq i\leq N'$, and write $N:=N_1+N'$. Since $\hat{f}_{\underline{\omega}^{N'}}$ admits a Lipschitz constant at most $L'$, we conclude that
    \begin{equation*}
        \operatorname{dist}_{\mathbb{T}^2\times\mathbb{S}^1}\!\left(\hat{f}_{\underline{\omega}^{N}}(x,v),(x_*,v_*)\right)\leq L'\epsilon_0<\epsilon,
    \end{equation*}as claimed.
\end{proof}

\subsection{Proof of Theorem \ref{thm: exponential mixing}}
We conclude this section by summarizing the proof of Theorem \ref{thm: exponential mixing}. First, observe that the basic assumptions (\ref{hypothese H_1})-(\ref{hypothese H_3}) are trivially satisfied by the randomized Chirikov vector field defined in \eqref{velocity 1} and \eqref{velocity 2}. To apply Proposition \ref{prop:discrete time mixing}, it remains to establish geometric ergodicity of the two-point process. 

In Sections \ref{subsec: small set and aperiodicity} and \ref{subsec: Topological irreducibility}, we verify the small-set condition, aperiodicity and topological irreducibility required by Harris theorem. It therefore suffices to construct a Lyapunov function $V$ for the two-point chain. By Lemma~\ref{lem:drift condition}, this is reduced to proving geometric ergodicity of the one-point and projective processes together with positivity of the top Lyapunov exponent. These properties are established in Sections \ref{subsec: small set and aperiodicity}--\ref{subsec: Topological irreducibility}. Combining these ingredients, the two-point process is uniformly geometrically ergodic, and Theorem \ref{thm: exponential mixing} follows for all discrete times $t=n\in\mathbb{N}$. 

To extend the result to all positive real times $t>0$, we write $t=n+\tau$ with $n\in2\mathbb{N}$ and $\tau\in[0, 2)$. Applying Proposition~\ref{prop:discrete time mixing}, we obtain
\begin{equation*}
    \left|\int_{\mathbb{T}^2}g(x)h\left(\phi_{n+\tau}^\nu(x)\right)\mathrm{d}x\right|\leq \tilde{D}_\nu(\underline{\omega})e^{-\lambda_sn}\|g\|_{\dot{H}^s}\|h\circ \phi_\tau^\nu\|_{\dot{H}^s}.
\end{equation*}Since $\tau\in [0,2)$ and the flow $\phi_\tau^\nu$ is a measure preserving diffeomorphism on $\mathbb{T}^2$, there exists a uniform constant $C_*>0$ such that
\begin{equation*}
    \sup_{\tau\in[0,2)}\|h\circ \phi_\tau^\nu\|_{\dot{H}^s}\le C_*\,\|h\|_{\dot{H}^s}
\end{equation*}Substituting $n=t-\tau$ and noting that $e^{-\lambda_sn} = e^{-\lambda_s(t-\tau)}\leq e^{\lambda_s}e^{-\lambda_st}$, we conclude that 
\begin{equation*}
    \left|\int_{\mathbb{T}^2}g(x)h\left(\phi_{n+\tau}^\nu(x)\right)\mathrm{d}x\right|\leq D_\nu(\underline{\omega})e^{-\lambda_st}g\|_{\dot{H}^s}\|h\|_{\dot{H}^s},
\end{equation*}where $D_\nu(\underline{\omega}):=\tilde{D}_\nu(\underline{\omega})\,C_*\,e^{\lambda_s}$. This completes the proof of Theorem \ref{thm: exponential mixing} for continuous time.

\vspace{0.5em}
\section*{Appendix}\label{Appendix}
\stepcounter{section}
\numberwithin{equation}{subsection}
\numberwithin{theorem}{subsection}
\setcounter{equation}{0}
\setcounter{theorem}{0}
\renewcommand\thesubsection{\Alph{subsection}}

\subsection{Proof of Theorem \ref{thm: general framework for quantitative exponential mixing}}\label{subsec: Proof of Theorem 1.1}
As indicated in Section \ref{subsec: Contraction Estimate for the Transition Kernel}, we now complete the proof of Theorem \ref{thm: general framework for quantitative exponential mixing} by utilizing the contraction estimate from Proposition \ref{prop: Contraction estimate for the transition kernel}. In what follows, $C$ denotes a generic constant independent of $A$, whose value may vary from line to line.
\begin{proof}[Proof of Theorem \ref{thm: general framework for quantitative exponential mixing}]
    We define $\mathbb{Z}_0^d:=\mathbb{Z}^d\setminus\{\mathbf{0}\}$ and denote the standard orthogonal basis for $L^2(\mathbb{T}^d)$ by $\left\{e_m(x)=e^{im\cdot x}:m\in\mathbb{Z}^d\right\}$. Let $\varphi,\psi\in H^1$ be mean-zero functions with the Fourier expansions
    \begin{equation*}
        \varphi=\sum_{m\in\mathbb{Z}_0^d}\hat{\varphi}_me_m,\qquad \psi=\sum_{m\in\mathbb{Z}_0^d}\hat{\psi}_me_m.
    \end{equation*}Fix a constant $\zeta>0$, to be explicitly chosen later. For any $m,m'\in\mathbb{Z}_0^d$, we define the random variables
    \begin{align*}
        N_{m,m'}&:=\max\!\left\{n\in\mathbb{N}:\left|\int e_m(x)e_{m'}\circ f_{\underline{\omega}}^n(x)\pi(\mathrm{d}x)\right|>e^{-\zeta n}\right\},\\
        M_0&:=\max\!\left\{|m|\vee|m'|:e^{\zeta N_{m,m'}}>|m|\ |m'|\right\},\\
        \hat{D}&:=\max_{|m|,|m'|\leq M_0} e^{\zeta N_{m,m'}}.
    \end{align*}Combining the union bound with Chebyshev inequality, we deduce that
    \begin{align*}
        \mathbb{P}(N_{m,m'}>l)&\leq \sum_{n>l}\mathbb{P}\!\left(\left|\int e_m(x)e_{m'}\circ f_{\underline{\omega}}^n(x)\pi(\mathrm{d}x)\right|>e^{-\zeta n}\right)\\[0.2em]
        &\leq \sum_{n>l}e^{2\zeta n}\mathbb{E}\left|\int e_m(x)e_{m'}\circ f_{\underline{\omega}}^n(x)\pi(\mathrm{d}x)\right|^2.
    \end{align*}Moreover, we can rewrite the second moment as
    \begin{equation*}
        \mathbb{E}\left|\int e_m(x)e_{m'}\circ f_{\underline{\omega}}^n(x)\pi(\mathrm{d}x)\right|^2=\int e_{m'}^{(2)}P^{(2),n}e_{m}^{(2)}\mathrm{d}\pi^{(2)},
    \end{equation*}where we set
    \begin{equation*}
        e_m^{(2)}(x,y):=e_m(x)\overline{e_{m'}}(y),\qquad \pi^{(2)}(\mathrm{d}x,\mathrm{d}y):=\pi(\mathrm{d}x)\pi(\mathrm{d}y).
    \end{equation*}Applying estimate \eqref{eq: general quantitative drift condition for P^n}, we bound the integral
    \begin{align*}
        \left|\int e_{m'}^{(2)}P^{(2),n}e_{m}^{(2)}\mathrm{d}\pi^{(2)}\right|&\leq\int\left|P^{(2),n}e_m^{(2)}\right|\mathrm{d}\pi^{(2)}\leq \int (1+\beta V)\!\left\|P^{(2),n}e_m^{(2)}\right\|_\beta\mathrm{d}\pi^{(2)}\\
        &\leq 2^{k+1}C_1^{(d-1)p}e^{-\frac{1}{l_0}\tau^2n}\left\|e_m^{(2)}\right\|_\beta\int(1+\beta V)\mathrm{d}\pi^{(2)}.
    \end{align*}Since the invariance of $\pi^{(2)}$ under $P^{(2)}$ yields $\int V\mathrm{d}\pi^{(2)}\leq \frac{C_3}{1-\gamma}$ (thus~$\int (1+\beta V)\mathrm{d}\pi^{(2)}\leq \frac{1}{1-\gamma}$) and $\left\|e_m^{(2)}\right\|_\beta\leq 1$, we obtain the uniform decay bound
    \begin{equation*}
        \left|\int e_{m'}^{(2)}P^{(2),n}e_{m}^{(2)}\mathrm{d}\pi^{(2)}\right|\leq \frac{1}{1-\gamma}2^{k+1}C_1^{(d-1)p}e^{-\frac{1}{l_0}\tau^2n}.
    \end{equation*}Hereafter, we enforce the condition
    \begin{equation*}
        2\zeta-\frac{1}{l_0}\tau^2<0.
    \end{equation*}Then we obtain
    \begin{equation}\label{eq: general tail probability for N_m,m'}
        \mathbb{P}(N_{m,m'}>l)\leq C2^{k+1}C_1^{(d-1)p}e^{\left(2\zeta-\frac{1}{l_0}\tau^2\right)l}.
    \end{equation}In particular, $N_{m,m'}<\infty$ almost surely and we have the estimate
    \begin{equation*}
        \left|\int e_m(x)e_{m'}\circ f_{\underline{\omega}}^n(x)\pi(\mathrm{d}x)\right|\leq e^{\zeta N_{m,m'}-\zeta n}.
    \end{equation*}Consequently,
    \begin{equation}\label{eq: general quantitative estimate for varphi(x)*psi-circ f_omega^n}
        \left|\int \varphi(x)\psi\circ f_{\underline{\omega}}^n(x)\pi(\mathrm{d}x)\right|\leq e^{-\zeta n}\!\sum_{m,m'\in\mathbb{Z}_0^2}|\hat{\varphi}_m||\hat{\psi}_{m'}|e^{\zeta N_{m,m'}}.
    \end{equation}By definition of $M_0$ and $\hat{D}$, we have
    \begin{equation*}
        e^{\zeta N_{m,m'}}\leq \hat{D}|m||m'|.
    \end{equation*}Hence, substituting this into \eqref{eq: general quantitative estimate for varphi(x)*psi-circ f_omega^n} gives
    \begin{align*}
        \left|\int \varphi(x)\psi\circ f_{\underline{\omega}}^n(x)\pi(\mathrm{d}x)\right|&\leq \hat{D}(\underline{\omega})e^{-\zeta n}\left(\sum_{m\in\mathbb{Z}_0^d}|m||\hat{\varphi}_m|\right)\!\left(\sum_{m\in\mathbb{Z}_0^d}|m'||\hat{\psi}_{m'}|\right)\\
        &\leq \hat{D}(\underline{\omega})e^{-\zeta n}\|\varphi\|_{\dot{H}^{\frac{d}{2}+2}}\|\psi\|_{\dot{H}^{\frac{d}{2}+2}}.
    \end{align*}To obtain the lower-regularity bound, we invoke the following approximation lemma
    \begin{lemma}[{\hspace{-0.05em}\cite[Lemma 7.1]{BBPS-22A}}]\label{lem: control H^s' via H^s}
        Let $0<s<s'$ and let $h\in H^s$ be a mean-zero function. Then, for any $\epsilon>0$, there exists a mean-zero function $h_\epsilon\in H^{s'}$ such that the following estimates hold
        \begin{enumerate}[label=(\arabic*), ref=(\arabic*)]
            \item $\|h_\epsilon\|_{L^2}\lesssim\|h\|_{L^2}$.\vspace{0.2em}
            \item $\|h_\epsilon-h\|_{L^2}\lesssim \epsilon \|h\|_{H^s}$.\vspace{0.2em}
            \item $\|h_\epsilon\|_{H^{s'}}\lesssim_{s,s'}\epsilon^{-\frac{s'-s}{s}}\|h\|_{H^s}$.
        \end{enumerate}
    \end{lemma}We apply this result with $s'=\frac{d}{2}+2$ and $s=1$. Since $\pi=\operatorname{Leb}$, fixing an $\epsilon>0$, we deduce that
    \begin{align*}
        \left|\int \varphi(x)\psi\circ f_{\underline{\omega}}^n(x)\pi(\mathrm{d}x)\right|&\leq \left|\int \varphi_{\epsilon}(x)\psi_{\epsilon}\circ f_{\underline{\omega}}^n(x)\pi(\mathrm{d}x)\right|+\|\psi\|_{L^2}\|\varphi-\varphi_\epsilon\|_{L^2}+\|\varphi_\epsilon\|_{L^2}\|\psi-\psi_{\epsilon}\|_{L^2}\\[0.2em]
        &\leq C\hat{D}(\underline{\omega})\!\left(e^{-\zeta n}\epsilon^{-\frac{d}{2}-1}+\epsilon\right)\|f\|_{\dot{H}^1}\|g\|_{\dot{H}^1}.
    \end{align*}Optimizing $\epsilon$ on the right-hand side by balancing the terms (i.e., setting $\epsilon=e^{-\frac{2\zeta}{d+4}n}$) yields the desired exponential mixing estimate
    \begin{equation*}
        \left|\int \varphi(x)\psi\circ f_{\underline{\omega}}^n(x)\pi(\mathrm{d}x)\right|\leq C\hat{D}(\underline{\omega})e^{-\frac{2\zeta}{d+4}n}\|\varphi\|_{\dot{H}^1}\|\psi\|_{\dot{H}^1}.
    \end{equation*}

    To complete the proof, we need to bound the moments of $\hat{D}$. We require $\zeta$ to be sufficiently small such that
    \begin{equation}\label{eq: general requirements for zeta}
        \frac{2\zeta-\frac{1}{l_0}\tau^2}{\zeta q}+2<0,\qquad \frac{5d}{2}+1+\frac{2\zeta-\frac{1}{l_0}\tau^2}{\zeta}<0. 
    \end{equation}For instance, one may choose 
    \begin{equation*}
        \zeta=\frac{\tau^2}{2l_0}\min\!\left\{\frac{1}{2(1+q)},\frac{1}{\frac{5d}{2}+3}\right\}.
    \end{equation*}Using \eqref{eq: general tail probability for N_m,m'}, we estimate the tail of $M_0$
    \begin{align*}
        \mathbb{P}(M_0>l)&\leq 2\sum_{m,m'\in\mathbb{Z}_0^d,\ |m|>l}\mathbb{P}\!\left(e^{\zeta N_{m,m'}}>|m||m'|\right)\\[0.2em]
        &\leq C2^{k+2}C_1^{(d-1)p}\sum_{m'\in\mathbb{Z}_0^d}|m'|^{\frac{2\zeta-\frac{1}{l_0}\tau^2}{\zeta}}\sum_{m\in\mathbb{Z}_0^d,\ |m|>l}|m|^{\frac{2\zeta-\frac{1}{l_0}\tau^2}{\zeta}}\\[0.2em]
        &\leq C2^{k+2}C_1^{(d-1)p}l^{d+\frac{2\zeta-\frac{1}{l_0}\tau^2}{\zeta}},
    \end{align*}where the last step uses the summability ensured by \eqref{eq: general requirements for zeta}. Hence $\mathbb{P}(M_0<\infty)=1$.

    Finally, we estimate $\mathbb{E}[\hat{D}^q]$ by partitioning over the values of $M_0$:
    \begin{align*}
        \mathbb{E}[\hat{D}^q]&=\sum_{l=1}^\infty \mathbb{E}\!\left(\mathbf{1}_{\{M_0=l\}}\max_{|m|,|m'|\leq l}e^{\zeta qN_{m,m'}}\right)\leq \sum_{l=1}^\infty\mathbb{P}(M_0=l)^{\frac{1}{2}}\!\left\|\max_{|m|,|m'|\leq l}e^{\zeta qN_{m,m'}}\right\|_{L^2}\\[0.2em]
        &\leq C2^{\frac{k}{2}+1}C_1^{(d-1)p/2}\sum_{l=1}^\infty l^{\frac{d}{2}+\frac{2\zeta-\frac{1}{l_0}\tau^2}{2\zeta}}\!\left(\sum_{|m|,|m'|\leq l}\!\left\|e^{\zeta qN_{m,m'}}\right\|_{L^2}\right)\\
        &\leq C2^{\frac{k}{2}+1}C_1^{(d-1)p/2}\!\left(1+\frac{\zeta q}{\frac{1}{l_0}\tau^2-2\zeta(1+q)}\right)^{\frac{1}{2}}\sum_{l=1}^\infty l^{\frac{5d}{2}+\frac{2\zeta-\frac{1}{l_0}\tau^2}{2\zeta}}.
    \end{align*}By the first condition in \eqref{eq: general requirements for zeta}, the exponent of $l$ is strictly less than $-1$, ensuring the sum converges. Thus, we conclude 
    \begin{equation*}
        \mathbb{E}[\hat{D}^q]\leq C2^{k/2}C_1^{(d-1)p/2},
    \end{equation*}which completes the proof.
\end{proof}

\subsection{Auxiliary tools} 
The following lemmas are essential for verifying the small set condition in Section \ref{subsec: Quantitative small set condition}. For detailed proofs, we refer 
the reader to \cite{Christ-85,son-25}.

\begin{lemma}[Quantitative inverse function theorem]\label{lem: Quantitative inverse function theorem}
    Assume that $F\colon \mathbb{R}^d\rightarrow \mathbb{R}^d$ is a $C^2$ function. Suppose that there exist positive constants $D,r$ such that
    \begin{equation*}\label{eq:assumption of inverse function theroem}
        \|F\|_{C^2}\leq D,\qquad |\det DF(0)|\geq r.
    \end{equation*}Then, there exist constants $C_1:=C(d)D^{-d}$ and $C_2:=C(d)D^{-2d+1}$ such that the following properties hold:
    \begin{enumerate}[label=(\arabic*), ref=(\arabic*)]
        \item $F$ is injective on $B(0,C_1r)$.
        \item $B\left(F(0),C_2r^2\right)\subset F\left(B\left(0,C_1r\right)\right)$.
        \item For all $x\in B(0,C_1r)$, we have that 
        \begin{equation*}
            \frac{1}{2}r\leq \left|\det DF(x)\right|\leq \frac{3}{2}r.
        \end{equation*}
    \end{enumerate}
    Here $C(d)$ is a constant depending on the dimension $d$.
\end{lemma}
\begin{lemma}[Quantitative implicit function theorem]\label{lem: Quantitative implicit function theorem}
    Let $G\colon\mathbb{R}^{d_1+d_2}\rightarrow\mathbb{R}^{d_2}$ be a $C^2$ function. Assume that there exist $(x_0,y_0)\in\mathbb{R}^{d_1}\times\mathbb{R}^{d_2},D\geq 1$ and $r>0$ such that
    \begin{equation*}
        G(x_0,y_0)=0,\quad \|G\|_{C^2}\leq D,\quad \left|\det D_yG(x_0,y_0)\right|\geq r.
    \end{equation*}Then, there exist a constant $C_2:=C(d_1+d_2)D^{-2(d_1+d_2)+1}$ and a $C^1$ function $H\colon B(x_0,C_2r^2)\rightarrow\mathbb{R}^{d_2}$ such that for all $x\in B(x_0,C_2r^2)$, it holds that 
    \begin{enumerate}[label=(\arabic*), ref=(\arabic*)]
        \item $H(x_0)=y_0$.
        \item $G\left(x,H(x)\right)=0$.
        \item\label{implicit function: det bound of implicit function} $\left|\det D_yG\left(x,H(x)\right)\right|\geq \frac{1}{2}r$.
    \end{enumerate}
\end{lemma}

\subsection{Application to the Pierrehumbert model}\label{sec: Application to the Pierrehumbert model}
In this appendix, we outline the application of Theorem \ref{thm: general framework for quantitative exponential mixing} to the Pierrehumbert model on~$\mathbb{T}^2$; see \cite{BCZG-23} for further details. Fix a parameter $A>0$ and let $\left\{\omega_i=\left(\omega_i^1,\omega_i^2\right)\right\}_{i\in\mathbb{N}}$ be an i.i.d sequence uniformly distributed on $\Omega_0:=[0,2\pi)^2$. The continuous RDS is driven by the one-step dynamics 
\begin{equation*}
    f_\omega(x):=f_{\omega^2}^V\circ f_{\omega^1}^H(x),
\end{equation*}where $f^H$ and $f^V$ denote the horizontal and vertical shear maps given by
\begin{equation*}
    f_{\beta}^{H}(x):=\begin{pmatrix}
        x_1 + A\sin\!\left(x_2 - \beta\right) \\
        x_2
    \end{pmatrix},\qquad 
    f_{\beta}^{V}(x):=\begin{pmatrix}
        x_1 \\
        x_2 + A\sin\!\left(x_1-\beta\right)
    \end{pmatrix}.
\end{equation*}

It is straightforward to see that the Pierrehumbert model naturally satisfies the fundamental structural conditions (\ref{hypothese H_0}). Consequently, applying Theorem \ref{thm: general framework for quantitative exponential mixing} reduces to validating Assumptions \ref{assumption: general uniform contraction}-\ref{assumption: general small set condition}. Given that the estimates required by Assumptions \ref{assumption: general derivative control} and \ref{assumption: general small set condition} have been previously established in \cite{son-25}, our remaining task is to verify the uniform contraction condition (Assumption \ref{assumption: general uniform contraction}). Unlike the Chirikov dynamics, the Pierrehumbert model exhibits a stronger form of non-degeneracy: the randomness from both independent phases $\omega^1$ and $\omega^2$ remains after taking derivatives. Leveraging this persistent randomness, we shall extract a one-step contraction estimate in the following proposition.

\begin{proposition}\label{prop:Pierrehumbert-one-step-contraction}
    There exist constants $A_0>0$ and $p,\gamma'\in \left(0,\frac{1}{2}\right)$ such that for all $A\geq A_0$, any $x\in\mathbb{T}^2$ and $v\in\mathbb{S}^1$, we have
    \begin{equation}
        \mathbb{E}\left\|D_xf_\omega(x)v\right\|^{-p}\leq \gamma',
    \end{equation}where the expectation is taken with respect to $\omega\in\Omega_0$.
\end{proposition}
\begin{proof}
    For any $x=(x_{0,1},x_{0,2})\in\mathbb{T}^2$ and $\omega=(\omega^1,\omega^2)\in\Omega_0$, the one-step Jacobian is given by
    \begin{equation*}
        D_xf_\omega=\begin{pmatrix}
            1 & C^H\\
            C^V & 1+C^HC^V
        \end{pmatrix},
    \end{equation*}where we define the random variables
    \begin{equation*}
        C^H=A\cos\!\left(x_{0,2}-\omega^1\right),\qquad C^V=A\cos\!\left(x_{0,1}+A\sin\!\left(x_{0,2}-\omega^1\right)-\omega^2\right).
    \end{equation*}It is straightforward to verify that $C^H$ and $C^V$ are independent and identically distributed. Specifically,
    \begin{equation*}
        C^V\overset{d}{=}A\cos U, \qquad U\sim \operatorname{Unif}[0,2\pi).
    \end{equation*}
    
    Given a unit vector $v=(v_{0,1},v_{0,2})\in\mathbb{S}^1$, the iterated tangent vector $v_1:=D_x f_{\omega}(x)v$ satisfies the recurrence relations
    \begin{equation*}
        v_{1,1}=v_{0,1}+C^Hv_{0,2},\qquad v_{1,2}=C^Vv_{1,1}+v_{0,2}.
    \end{equation*}Proceeding as in the proof of Lemma \ref{lem: uniform derivative drift condition}, we obtain the moment bound 
    \begin{equation}\label{eq: estimates for v_1,1}
        \mathbb{E}\!\left[|v_{ 1,1}|^{-p}\right]\leq C_p.
    \end{equation}
    
    By conditioning on $v_{1,1}$, we have 
   \begin{align}
        \mathbb{E}\big[\|v_1\|^{-p}\big]&=\mathbb{E}\!\left(\mathbb{E}\!\left[(v_{1,1}^2+v_{1,2}^2)^{-p/2}\bigm|v_{1,1}\right]\right)\notag =\mathbb{E}\!\left(|v_{1,1}|^{-p}\mathbb{E}\!\left[\big(1+(v_{1,2}/v_{1,1})^2\big)^{-p/2}\bigm|v_{1,1}\right]\right).\label{eq: law of total expectation}
    \end{align}
    Observe that the coordinate ratio expands as
    \begin{equation*}
        \frac{v_{1,2}}{v_{1,1}}=C^V+\frac{v_{0,2}}{v_{1,1}}.
    \end{equation*}Since $v_{1,1}$ is measurable with respect to $\omega^1$ only, conditionally on $v_{1,1}$, $C^V$ still has the same distribution as $A\cos U$. Therefore, applying Claim \ref{claim: control of a+b cos-theta} yields
    \begin{align*}
        \mathbb{E}\!\left[\left(1+(v_{1,2}/v_{1,1})^2\right)^{-p/2} \Bigm| v_{1,1}\right]
        &\leq \mathbb{E}\!\left[\left|v_{1,2}/v_{1,1}\right|^{-p} \bigm| v_{1,1}\right]=\mathbb{E}\!\left[\left|A\cos U+v_{0,2}/v_{1,1}\right|^{-p} \bigm| v_{1,1}\right]\\
        &\leq C_p A^{-p}.    
    \end{align*}Combining this uniform bound with \eqref{eq: estimates for v_1,1}, we conclude that
    \begin{equation*}
        \mathbb{E}\|v_1\|^{-p}\leq \mathbb{E}\!\left(|v_{1,1}|^{-p}C_pA^{-p}\right)\leq C_p^2A^{-p}<\frac{1}{2},
    \end{equation*}provided that the parameter $A$ is chosen sufficiently large.
\end{proof}

\noindent {\bf Acknowledgments.} The authors are supported by NSFC (Nos. 12231002). We are sincerely grateful to Prof. Yong Liu for his encouragement, insightful suggestions, and stimulating discussions, all of which greatly benefited this manuscript.

\normalem
\bibliographystyle{abbrv}
\bibliography{References}

\end{document}